\documentclass[11pt,english]{article}

\usepackage[margin = 2.5cm]{geometry}

\usepackage{amsthm}
\usepackage{amsmath}
\usepackage{amssymb}
\usepackage{mathtools}
\usepackage{graphicx}
\usepackage{xcolor}
\usepackage[hidelinks, bookmarks = false]{hyperref}
\usepackage{enumitem}
\usepackage{bm}

\usepackage{caption}% http://ctan.org/pkg/caption
\usepackage[square,sort,comma,numbers]{natbib} 
\usepackage{setspace}
\usepackage{cleveref}
\theoremstyle{plain}

\newtheorem*{thm*}{Theorem}
\newtheorem{thm}{Theorem}[section]
\crefname{thm}{Theorem}{Theorems}
\Crefname{thm}{Theorem}{Theorems}

\newtheorem*{lem*}{Lemma}
\newtheorem{lem}[thm]{Lemma}
\crefname{lem}{Lemma}{Lemmas}
\Crefname{lem}{Lemma}{Lemmas}

\newtheorem*{claim*}{Claim}
\newtheorem{claim}[thm]{Claim}
\crefname{claim}{Claim}{Claims}
\Crefname{claim}{Claim}{Claims}

\newtheorem{prop}[thm]{Proposition}
\crefname{prop}{Proposition}{Propositions}
\Crefname{prop}{Proposition}{Propositions}

\newtheorem{cor}[thm]{Corollary}
\crefname{cor}{Corollary}{Corollaries}
\Crefname{cor}{Corollary}{Corollaries}

\newtheorem{conj}[thm]{Conjecture}
\crefname{conj}{Conjecture}{Conjectures}
\Crefname{conj}{Conjecture}{Conjectures}

\newtheorem{qn}[thm]{Question}
\crefname{qn}{Question}{Questions}
\Crefname{qn}{Question}{Questions}

\newtheorem{obs}[thm]{Observation}
\crefname{obs}{Observation}{Observations}
\Crefname{obs}{Observation}{Observations}

\crefname{ex}{Example}{Examples}
\Crefname{ex}{Example}{Examples}

\theoremstyle{definition}

\crefname{prob}{Problem}{Problems}
\Crefname{prob}{Problem}{Problems}

\newtheorem{defn}[thm]{Definition}
\crefname{defn}{Definition}{Definitions}
\Crefname{defn}{Definition}{Definitions}

\theoremstyle{remark}
\newtheorem*{rem}{Remark}
\crefname{rem}{Remark}{Remarks}
\Crefname{rem}{Remark}{Remarks}

\usepackage{xpatch}
\xpatchcmd{\proof}{\itshape}{\normalfont\proofnamefont}{}{}
\newcommand{\proofnamefont}{}
\renewcommand{\proofnamefont}{\bfseries}

\usepackage{framed}

\newcommand{\remove}[1]{}

\newcommand{\ceil}[1]{
    \lceil #1 \rceil
}
\newcommand{\floor}[1]{
    \lfloor #1 \rfloor
}

\newcommand{\rhat}{\hat{r}}
\newcommand{\eps}{\varepsilon}

\newcommand{\F}{\mathcal{F}}
\newcommand{\T}{\mathcal{T}}
\newcommand{\G}{\mathcal{G}}
\newcommand{\HH}{\mathcal{H}}
\newcommand{\A}{\mathcal{A}}

\newcommand{\N}{\mathbb{N}}

\newcommand{\sarrow}{\stackrel{s}{\longrightarrow}}
\newcommand{\family}{\A}

\DeclareMathOperator{\dist}{dist}
\DeclareMathOperator{\LG}{LG}

\title{Size-Ramsey numbers of powers of hypergraph trees and long subdivisions}

\author{
	Shoham Letzter\thanks{
	Department of Mathematics, 
	University College London, 
	Gower Street, London WC1E~6BT, UK. 
	Email: \texttt{s.letzter}@\texttt{ucl.ac.uk}. 
	Research supported by the Royal Society.
	}
	\and
	Alexey Pokrovskiy \thanks{
	Department of Mathematics, 
	University College London, 
	Gower Street, London WC1E~6BT, UK. 
	Email: \texttt{a.pokrovskiy}@\texttt{ucl.ac.uk}.
	}
	\and
	Liana Yepremyan\thanks{ 
	Deparment of Mathematics, London School of Economics, London WC2A 2AE, UK.
	Email: \texttt{l.yepremyan}@\texttt{lse.ac.uk}. Research supported by Marie Sklodowska Curie Global Fellowship, H2020-MSCA-IF-2018:846304.
	}
}

\date{}

\begin{document}

\maketitle

\begin{abstract}

	\setlength{\parskip}{\medskipamount}
	\setlength{\parindent}{0pt}
	\noindent
	
	The \emph{$s$-colour size-Ramsey number} of a hypergraph $H$ is the minimum number of edges in a hypergraph $G$ whose every $s$-edge-colouring contains a monochromatic copy of $H$. 
	We show that the $s$-colour size-Ramsey number of the $t$-power of the $r$-uniform tight path on $n$ vertices is linear in $n$, for every fixed $r, s, t$, thus answering a question of Dudek, La Fleur, Mubayi, and R\"odl (2017).

	In fact, we prove a stronger result that allows us to deduce that powers of bounded degree hypergraph trees and powers of `long subdivisions' of bounded degree hypergraphs have size-Ramsey numbers that are linear in the number of vertices. This extends and strongly generalises recent results about the linearity of size-Ramsey numbers of powers of bounded degree trees and of long subdivisions of bounded degree graphs.
	
\end{abstract}

\section{Introduction} \label{sec:intro}

	For two hypergraphs $G$ and $H$ and an integer $s \ge 2$ write $G \sarrow H$ if in every $s$-edge-colouring of $G$ there is a monochromatic copy of $H$. The $s$-colour \emph{size-Ramsey number} of a hypergraph $H$, denoted by $\rhat_s(H)$, is the minimum number of edges in a hypergraph $G$ satisfying $G \sarrow H$. Namely,
	\begin{equation*}
		\rhat_s(H) = \min\{e(G): G \sarrow H\}.
	\end{equation*}
	When $s = 2$, we often omit the subscript $2$, and refer to the $2$-colour size-Ramsey number of $H$ as, simply, the \emph{size-Ramsey number} of $H$.

	The notion of size-Ramsey numbers of graphs was introduced by Erd\H{o}s, Faudree, Rousseau and Schelp \cite{erdos1978size} in 1978. It is an interesting and well-studied variant of the classical $s$-colour \emph{Ramsey number} of $H$, denoted by $r_s(H)$, which is be defined as
	\begin{equation*}
		r_s(H) = \min\{|G| : G \sarrow H\},
	\end{equation*}
	i.e.\ it is the minimum number of vertices in a graph $G$ such that $G \sarrow H$. The study of Ramsey numbers, especially for $s = 2$, is one of the most central and well-studied topics in combinatorics.  In this paper we study size-Ramsey numbers of graphs and hypergraphs, an active field of research in recent years. 
	
	One of the earliest results regarding size-Ramsey numbers of graphs, obtained by Beck \cite{beck1983size} in 1983, asserts that the size-Ramsey number of a path is linear in its length; more precisely, Beck showed that $\rhat(P_n) \le 900 n$ for large $n$. The problem of determining the size-Ramsey number of a path has been the focus of many papers \cite{bal2019new,beck1983size,bollobas2001random,bollobas1986extremal,dudek2015alternative,dudek2017some,letzter2016path}, and the currently best-known bounds are as follows:
	\begin{equation*}
		(3.75 + o(1))n \le \rhat(P_n) \le 74n,
	\end{equation*}
	where the lower bound is due to Bal and DeBiasio \cite{bal2019new} and the upper bound is by  Dudek and Pra{\l}at \cite{dudek2017some}. One can easily generalise Beck's arguments to the multicolour setting, showing that $\rhat_s(P_n) = O_s(n)$. This multicolour variant has received its own fair share of attention \cite{dudek2017some,krivelevich2019long,bal2019new,dudek2018note}; and currently the best-known bounds are as follows, where the the lower bound is due to Dudek and Pra{\l}at \cite{dudek2017some} and  the upper  bound is by Krivelevich \cite{krivelevich2019long}.
	\begin{equation*}
		\rhat_s(P_n) = \Omega(s^2 n), \qquad \rhat_s(P_n) = O\big((s^2 \log{s}) n\big).
	\end{equation*}
    In contrast, the study of size-Ramsey numbers of hypergraphs was only initiated in 2017 by Dudek, La Fleur, Mubayi and R\"odl \cite{dudek2017size}. One of the first problems proposed in their paper is the following  generalisation of Beck's result \cite{beck1983size} regarding size-Ramsey numbers of paths. The \emph{$r$-uniform tight path} on $n$ vertices, denoted by $P_n^{(r)}$, is the $r$-uniform hypergraph on vertices $[n]$ whose edges are all sets of $r$ consecutive elements in $[n]$. Observe that $P_n^{(2)}$ is the path $P_n$ on $n$ vertices. The authors of \cite{dudek2017size} ask if, similarly to the graph case, $\rhat\big(P_n^{(r)}\big) = O(n)$. This was answered affirmatively for $r = 3$ by Han, Kohayakawa, Letzter, Mota and Parczyk \cite{han2021size}. For $r \ge 4$, the best-known bound prior to our work was $\rhat(P_n^{(r)}) = O\big( (n \log n)^{r/2}\big)$, due to Lu and Wang \cite{lu2018size}.  The problem of determining if $\rhat\big(P_n^{(r)}\big) = O(n)$ was the main motivation  of our work, and, as one of the consequences of our main result (\Cref{thm:main} which is stated below), we settle this problem.

	\begin{thm} \label{thm:main-path0}
		Fix integers $r, s \ge 1$. Then $\rhat_s\big(P_n^{(r)}\big)=O(n)$.
	\end{thm}
	
    In fact, we show that the same result holds also for so-called powers of tight paths. While there is no standard definition of powers of hypergraphs in the literature, it is natural to define powers of tight paths as follows; this definition was used previously in \cite{bedenknecht2019powers}.
    Define the \emph{$t$-power} of an $r$-uniform tight path on $n$ vertices, denoted by $\big(P_n^{(r)}\big)^t$, to be the $r$-graph on vertices $[n]$ whose edges are the $r$-sets of vertices contained in intervals in $[n]$ of length $r+t-1$.  The following theorem, strengthening \Cref{thm:main-path0}, also follows from our main result.

	\begin{thm} \label{thm:main-path}
		Let $r, s, t \ge 1$ be integers. 
		The $s$-colour size-Ramsey number of $\big(P_n^{(r)}\big)^t$ is $O(n)$.
	\end{thm}
	
    The same problem for graphs was studied previously. For a graph $G$, its \emph{$t$-power} is the graph on $V(G)$ where two vertices are adjacent if the distance between them in $G$ is at most $t$. Following a rather long gap after the linearity of the size-Ramsey number of cycles was established by Haxell, Kohayakawa and {\L}uczak \cite{haxell1995induced}, Conlon~\cite{conlon2016} asked if powers of paths also have linear size-Ramsey numbers. Clemens, Jenssen, Kohayakawa, Morrison, Mota, Reding and Roberts \cite{clemens2019size}  proved this to be the case for two colours and Han, Jenssen, Kohayakawa, Mota and Roberts \cite{han2020multicolour} for $s \ge 2$ colours. 

    Another family of graphs whose size-Ramsey numbers have been studied extensively is the family trees. An influential result of Friedman and Pippenger \cite{friedman1987expanding} implies that bounded degree trees have linear size-Ramsey number too. 
    Here the bounded degree requirement is essential, since, for example, the double star, obtained by joining by an edge the centres of two disjoint stars on $n$ vertices, has size-Ramsey number which is quadratic in $n$.  Following the line of research which stemmed from Conlon's question,  Kam\v{c}ev, Liebenau, Wood and Yepremyan \cite{kamcev2019size} showed that powers of bounded degree trees have  linear size-Ramsey number, for two colours. Berger, Kohayakawa, Maesaka, Martins, Mendon\c{c}a, Mota and Parczyk \cite{berger2020size} showed that the same holds for an arbitrary number of colours. More precisely,  if we let  $\T_{n, d}$ to be the family of trees on $n$ vertices with maximum degree at most $d$, then by these results we know that $\rhat_s(T^t) = O(n)$ for every $T \in \T_{n,d}$, where $s, t, d$ are fixed.  
    
    There are various ways one can think of trees in the hypergraph setting. Inspired by Dudek, La Fleur, Mubayi and R\"odl  \cite{dudek2017size} we study hypergraph trees defined as follows.

	\begin{defn} \label{def:hypergraph-tree}
		An \emph{$r$-uniform tree} is an $r$-unifrom hypergraph with edges $\{e_1, \ldots, e_m\}$ such that for every $i \in \{2, \ldots, m\}$ we have $\big|e_i \cap (\bigcup_{1 \le j < i} e_j)\big| \le r-1$ and $e_i \cap (\bigcup_{1 \le i < j} e_j) \subseteq e_{i_0}$ for some $i_0 \in [i-1]$.\footnote{It would be more natural to require that  $e_i\cap \bigcup_{1 \le j < i} e_j \neq \emptyset$, or to refer to hypergraphs as in this definition as `forests', but we stick to this definition and notation to be consistent with \cite{dudek2017size}.}
	\end{defn}

	This generalises two notions of hypergraph trees: \emph{tight trees}, where the intersection $e_i \cap (\bigcup_{1 \le j < i} e_j)$ is required to have size \emph{exactly} $r-1$ rather than at most $r-1$; and \emph{$\ell$-trees}, for $\ell \in [r-1]$,  where the same intersection is required to have size at most $\ell$. The latter definition appears in Section 2.2 in \cite{dudek2017size}, where the authors ask if there are $\ell$-trees on $n$ vertices whose size-Ramsey number is $\Omega(n^{\ell+1})$. Here we answer this question negatively for bounded degree hypergraph trees, for all colours $s\geq 2$.

    \begin{thm} \label{thm:main-hypergraph-tree-0}
		Let $r, s, d \ge 1$ be integers. Every $r$-uniform tree $\T$ on $n$ vertices with maximum degree at most $d$ satisfies $\rhat_s(\T) = O(n)$.
	\end{thm}

    Our method applies to the powers of hypergraph trees as well. As far as we know, there is no standard notion of a power of a hypergraph. We thus make the following definition, which is reminiscent of the more natural notion of a power of a tight path (introduced before \Cref{thm:main-path}).

    \begin{defn} \label{def:power}
        Given an $r$-uniform hypergraph $\HH$, its \emph{$t$-power}, denoted by $\HH^t$, is the $r$-uniform hypergraph on $V(\HH)$ whose edges are $r$-sets of vertices that are contained in some tight path in $\HH$ on at most $r+t-1$ vertices.
    \end{defn}

    The following is a generalisation of the results of \cite{kamcev2019size} and \cite{berger2020size} for hypergraph trees. 

	\begin{thm} \label{thm:main-hypergraph-tree}
		Let $r, s, d, t \ge 1$ be integers. Every $r$-uniform tree $\T$ on $n$ vertices with maximum degree at most $d$ satisfies $\rhat_s(\T^t) = O(n)$.
	\end{thm}
	
    We deduce   \cref{thm:main-hypergraph-tree} from a stronger result,  \Cref{thm:main-tree} below (see  \Cref{sec:reductiontrees} for the deduction). To state the latter we need the following notation. For a graph $G$ and integer $r \ge 1$, denote by $K_r(G)$ the $r$-uniform hypergraph on $V(G)$ whose edges are $r$-cliques in $G$, and write $G[K_b]$ for the  \emph{complete blow-up} of $G$ where each vertex is replaced by a $b$-clique. 
	
	\begin{thm} \label{thm:main-tree}
		Let $r, s, d, t \ge 1$ be integers. 
		Every $T \in \T_{n, d}$ satisfies $\rhat_s(K_r(T^t)) = O(n)$.
	\end{thm}

    In light of the results above, it might be tempting to think that the size-Ramsey number of any bounded degree graph or hypergraph should be linear in its order, and, indeed, this was suggested by Beck \cite{beck1983size} for graphs. However, it turns out not to be the case; R\"odl and Szemer\'edi \cite{rodl2000size} constructed a sequence $(H_n)$ where $H_n$ is an $n$-vertex graph with maximum degree $3$ such that $\rhat(H_n) = \Omega\big(n (\log n)^{1/60}\big)$. This construction was also generalized to hypergraphs in~\cite{dudek2017size}.  
    
    It is easy to see that the size-Ramsey number of a bounded degree graph on $n$ vertices is $O(n^2)$. Indeed, this follows from the linearity of Ramsey numbers of bounded degree graphs, established by Chvat\'al, R\"odl, Szemer\'edi and Trotter \cite{chvatal1983ramsey}, and by the observation that $\rhat(H)\leq {r(H) \choose 2}$. Kohayakawa, R\"odl, Schacht and Szemer\'edi \cite{kohayakawa2011sparse} provided a significant improvement of this easy bound: they showed that, if $H$ is an $n$-vertex graph with maximum degree $\Delta$, then $\rhat(H) = O(n^{2-1/\Delta} (\log n)^{1/\Delta})$; this is the best-known upper bound to date on the size-Ramsey numbers of bounded degree graphs.

	The above upper bound is quite far from linear. One way to obtain graphs whose size-Ramsey numbers are closer to linear is to subdivide the edges of a bounded degree graph. The \emph{$q$-subdivision} of a graph $H$ is the graph obtained by subdividing each edge of $H$ by $q-1$ vertices; more precisely, the $q$-subdivision is obtained by replacing each edge $uv$ by a path $P_{uv}$ of length $q$ and ends $u$ and $v$, whose the interior vertices are unique to $P_{uv}$. Dragani\'c, Krivelevich and Nenadov \cite{draganic2021size} recently showed that, for fixed $\Delta$ and $q$, the size-Ramsey number of the $q$-subdivision of an $n$-vertex graph with maximum degree $\Delta$ is bounded by $O(n^{1+1/q})$ (improving on a result of Kohayakawa, Retter and R\"odl \cite{kohayakawa2019size}). This is close to tight if the host graph is a random graph. However, it is unclear if this bound is anywhere near tight in general, as the only known superlinear bound on size-Ramsey number of a family of bounded degree graph is the aforementioned bound of $\Omega(n (\log n)^{1/60})$, due to R\"odl and Szemer\'edi. Dragani\'c, Krivelevich and Nenadov also showed \cite{draganic2021rolling} that if $H$ is the $q$-subdivision of a bounded degree graph, where $|H| = n$ and $q \ge c \log n$ for a large constant $c$, then $\rhat_s(H) = O(n)$, thus confirming a conjecture of Pak \cite{pak2002mixing} from 2002.
    We generalize these results on subdivisions to the hypergraph in two aspects: we establish the linearity of powers of such graphs, and we generalise the result to hypergraphs.

    \begin{thm} \label{thm:main-subdivision}
		Let $r, s, d, t \ge 1$ be fixed integers, and let $n$ be large.
		If $H$ is a graph on $n$ vertices, obtained from a graph with maximum degree at most $d$ by subdividing each edge at least $10t \log n$ times, then $\rhat_s(K_r(H^t)) = O(n)$.
	\end{thm}

	\subsection{Our main result}
        We now state our main result in its full generality. Let $\family_{n, d, \ell}$ be the family of graphs on $n$ vertices with maximum degree at most $d$, that can be obtained from a singleton graph by successively either adding a new vertex and joining it by an edge to an existing one, or connecting two existing vertices by a path of length at least $\ell$ whose interior vertices are new.

		\begin{thm} \label{thm:main}
			Fix integers $r, s, t, d \ge 1$, and let $n$ be sufficiently large. 
			There exists an $r$-uniform hypergraph $\G$ with $O(n)$ edges such that $\G \sarrow K_r(H^t)$ for every $H \in \family_{n, d, 10t\log n}$.
		\end{thm}

		In fact, our proof shows that $\G$ can be taken to be $K_r(G^{k}[K_B])$, where $G$ is a bounded degree \emph{expander} (i.e.\ there is an edge between every two large sets of vertices) on $\Theta(n)$ vertices, and $k$ and $B$ are large constants.  
		
		Let us see first how \Cref{thm:main} implies all the previously mentioned results.  Recall that $P_n^{(r)}$ is the $r$-uniform tight path on $n$ vertices. It readily follows from \Cref{thm:main} that the $s$-colour size-Ramsey number of $P_n^{(r)}$ is linear in $n$. Indeed, take $H = P_n$, $t = r$, and observe that $K_r\big(P_n^{r-1}\big) \cong P_n^{(r)}$. This proves \Cref{thm:main-path0}. Observe that $\big(P_n^{(r)}\big)^t \cong K_r(P_n^{r+t-2})$, which implies that \cref{thm:main-path} also follows from \Cref{thm:main}. For \cref{thm:main-tree}, note that $\T_{n,d} \subseteq \family_{n, d, \ell}$ for any $\ell$, so \Cref{thm:main-tree} is a special case of \Cref{thm:main}. And, finally, if $H$ is a graph on $n$ vertices, obtained from a graph with maximum degree at most $d$ by subdividing each edge at least $10t \log n$ times, then $H \in \family_{n, d, \ell}$ with $\ell = 10t \log n$, implying that \cref{thm:main-subdivision} also follows from \Cref{thm:main}.

		Our method  has several novelties. First, we prove that, given a bounded degree hypergraph $\HH$, every $s$-colouring of $\HH[K_B]$ (the complete blow-up of $\HH$ by $B$-cliques )  contains a copy of $\HH[K_b]$ where edges of the same `type' have the same colour, provided that $B$ is sufficiently large with respect to $b$. Two edges have the same `type' if they intersect each $b$-clique in the same number of vertices. This is extremely useful for us. For example, it allows us to assume that there are `many' copies of each vertex, essentially all vertices in the same $b$-clique are interchangeable. In particular, the task of finding, say, a monochromatic tight path of length $n$ can be reduced to that of finding a monochromatic tight \emph{walk} of length $n$ where each vertex is repeated $O(1)$ times.
		
		Second, we introduce the notion of `ordered trees' which will be used to model cliques in the host hypergraph with some additional structure. We define an \emph{ordered tree} to be, roughly speaking, a rooted tree equipped with an ordering of its leaves that corresponds to a planar drawing. We prove two Ramsey-type results for ordered trees and forests which serve as stepping stones for all our further proofs. 
		
        A final, rather subtle idea, is the use of overlapping ordered trees associated with each vertex of a host hypergraph as in \Cref{thm:main}. Our proof proceeds by sequentially refining the structure of the host hypergraph based on these trees, using the results from the previous paragraphs. Each such step requires extensive ‘trimming’ of the ordered trees associated with each vertex. The expansion properties of the host hypergraph are maintained throughout the process using the overlapping properties of the ordered trees.
    		
		The above ideas and results provide us with the means to prove the linearity of size-Ramsey numbers of a general family of graphs and hypergraphs in a unified and relatively clean way. This is in contrast with previous results on this problem, whose methods were more ad hoc, especially for the multicolour setting (with $s \ge 3$). 
		Thus, we think that the methods developed in this paper will be useful for the systematic study of size-Ramsey numbers of graphs and hypergraphs.

	\subsection*{Structure of the paper}

		In the next section we present the main ideas of the proof. \Cref{sec:ramsey-expanders} contains one of the main ingredients of the proof,  \Cref{lem:ramsey-expanders}, which is a Ramsey-type result for powers of expanders; two of the proofs in this sections are delayed to \Cref{sec:appendix,sec:expanding-thm}.
		In~\Cref{sec:colouredorderedtrees}, we define the notion of ordered trees and prove two Ramsey-type results for ordered trees and forests. 
		\Cref{sec:hanging-trees} contains various auxiliary results about hypergraphs whose every vertex has an ordered tree associated with it.
		In~\Cref{sec:main-proof} we put together all auxiliary results to first prove the main component of our proof, \Cref{thm:main-auxiliary}, and then deduce our main result, \Cref{thm:main}. 
		The latter task is accomplished via a reduction (\Cref{lem:reduction}) whose proof we postpone to \Cref{sec:subdivisions}, since it is long but not very novel.
		Finally, in \Cref{sec:reductiontrees} we deduce \Cref{thm:main-hypergraph-tree} from the main result, more precisely from \cref{thm:main-tree}.

\section{Proof Overview} \label{sec:overview}

	In this section we describe some of the main ingredients of our proof. 
	Recall that for a graph $G$ and integers $B$ and $r$, we define $G[K_B]$ to be the complete blow-up of $G$ obtained by replacing each vertex of $G$ by a $B$-clique, and $K_r(G)$ is the $r$-graph on vertex set $V(G)$ whose edges are the $r$-cliques in $G$. We say that a statement holds when $\alpha \ll \beta$, for two constants $\alpha$ and $\beta$, if for any choice of $\alpha$ there exists some $\beta_0 \geq \alpha$ such that for all $\beta \geq \beta_0$ the statement holds.

	We first briefly describe previous proofs of similar results \cite{bedenknecht2019powers,berger2020size,clemens2019size,han2020multicolour}. Let $H$ be a power of a path or bounded degree tree with $|H| = n$. Take $G$ to be an expander\footnote{Various notions of expansions were used, typically stronger than what we use.} with $O(n)$ vertices and maximum degree $O(1)$, and consider the host graph $G' = G^k[K_B]$, where $k$ and $B$ are large constants. Fix an $s$-colouring of $G'$; it now suffices to find a monochromatic copy of $H$. The vertices of $G'$ naturally split into $B$-cliques $B(u)$, for $u \in V(G)$. The first step in the above proofs is to apply Ramsey's theorem to each clique $B(u)$ to find large monochromatic subcliques $B'(u) \subseteq B(u)$. Say that the majority colour of the subcliques is blue, and focus on the subgraph of $G'$ induced by the union of blue subcliques $B'(u)$. The proof now proceeds by either embedding a blue copy of $H$ using the blue subcliques and blue edges between them, or, if this fails, leverage the sparsity of blue edges between subcliques to embed $H$ in another colour. The latter part tends to include ad hoc arguments, which becomes trickier when the number of colours is larger than $2$. 

	Our setup is similar to the one described above.
	Let $\G = K_r(G^k[K_B])$, where $G$ is an expander\footnote{Here an \emph{expander} is a graph where there is an edge between every two large sets of vertices.} on $O(n)$ vertices with maximum degree $O(1)$, and $k$ and $B$ are large constants. In particular, as above, the vertices of $\G$ consist of pairwise disjoint cliques $B(u)$ of size $B$, for $u \in V(G)$. The first new ingredient in our proof is the following (see \Cref{lem:monochromatic-blowup}): we show that there are large subcliques $B'(u) \subseteq B(u)$ such that any two edges in $\bigcup_{u} B'(u)$ of the same \emph{type} have the same colour. Here two edges have the same \emph{type} if they intersect each $B'(u)$ in the same number of vertices. In particular, we recover the above property that each $B'(u)$ is a monochromatic clique, but we also know quite a lot more: for $r = 2$ we additionally have that all edges between $B'(u)$ and $B'(v)$ have the same colour; and for $r = 3$ we have, e.g., that all edges with one vertex in $B'(u)$ and two vertices in $B'(v)$ have the same colour, for any edge $uv$ in $G^k$.

	An immediate advantage of this is that, in order to find a monochromatic copy of an $r$-graph $\HH$, it suffices to find a monochromatic homomorphic copy $\HH'$ of $\HH$ in every $s$-colouring of $K_r(G^k)$, such that each vertex in $\HH'$ is the image of not too many vertices in $\HH$. To see this, suppose for convenience that $r = 2$ and consider the auxiliary colouring of $G^k$, obtained by colouring $uv$ by the colour of the edges between $B'(u)$ and $B'(v)$. Now given a monochromatic $\HH'$ in $G^k$ as above, it can be `lifted' to a monochromatic copy of $\HH$ in $G^k[K_B]$, by replacing the copies of a vertex $v$ by distinct vertices in $B'(v)$. 

	For the rest of this discussion, we focus on the case where $\HH$ is a tight $r$-uniform path on $n$ vertices. As explained in the previous paragraph, it suffices to show that in every $s$-colouring of $K_r(G^k)$ there is a monochromatic homomorphic copy $\HH'$ of $\HH$ such that no vertex in $\HH'$ is the image of too many vertices in $\HH$ (so $\HH'$ is a tight $r$-uniform walk on $n$ vertices where no vertex repeats too many times).

	In the remainder of this proof sketch we use the notion of \emph{ordered trees}, which we define to be rooted trees whose leaves come with a prescribed `natural ordering'; see \Cref{def:ordered-tree} and \Cref{fig:ordered-tree}. 

	To find the required copy of $\HH$, we iteratively assign an ordered tree $T_j(u)$ to each vertex $u$, for $j \in \{0, \ldots, h\}$, where $h$ is a large constant. The tree $T_0(u)$ is simply the singleton tree $\{u\}$, and for $j \ge 1$ the tree $T_j(u)$ is an ordered tree, rooted at $u$, which is a $d_j$-ary tree of height $j$, and whose vertices are vertices of $G$. Here $d_1, \ldots, d_h$ is a decreasing sequence (with $d_j \ll d_{j-1}$), so the trees $T_j(u)$ get taller with $j$ but their degrees decrease (so they get `narrower'). Our aim is to use these trees as auxiliary structures to locate the desired monochromatic $\HH'$. An important step towards this goal is a generalisation of the above step that produced large subcliques where edges of the same type have the same colour. Roughly speaking, given an appropriate $r$-graph on $\bigcup_u L(T_j(u))$ (where $L(T)$ is the set of leaves in $T$) with an $s$-colouring inherited from $K_r(G^k)$, we may assume (by appropriately `trimming' the trees $T_j(u)$) that edges of the same `type' have the same colour. Here two sets $A$ and $B$ of leaves in $T_j(u)$ have the same \emph{type} if the minimal subtrees of $T_j(u)$ with leaves $A$ and $B$, respectively, are isomorphic.\footnote{The reason we take the trees $T_j(u)$ to be ordered is to allow such a statement to be true; a similar statement for unordered trees does not hold.}
	One can obtain a similar definition of a type of a set $A$ that consists of leaves from multiple trees $T_j(u)$.

	Given a collection of trees as above for some $j$, we consider the graph $G^{k_{j+1}}$, where $k_1, \ldots, k_h$ is a decreasing sequence of suitably chosen constants. We equip this graph with an auxiliary $s'$-colouring (with $s' \gg s$), defined as follows. Roughly speaking, we colour an edge $uv$ by $(c, S)$ if there is a $c$-coloured `connector' between $T_j(u)$ and $T_j(v)$ of type $S$; if there is no such connector for any $c$ and $S$, we colour $uv$ grey. It turns out that there are two possible outcomes (see \Cref{cor:ramsey-expanders}): either there is a long monochromatic non-grey path; or there is a collection of large disjoint grey cliques that covers almost all of the vertices. In the former case, the monochromatic path can be lifted to the required monochromatic $\HH'$, and in the latter case, we use the grey cliques to define new taller trees $T_{j+1}(v)$. This process will not run forever, because, very roughly speaking, when the trees are tall enough, the required connectors must exist.

	What do we mean by a `connector'? A \emph{connector} for an edge $uv$ in $G^{k_{j+1}}$ consists of disjoint sets $X, Y, Z$ of size $r$ where $X \cup Y$ and $Y \cup Z$ are monochromatic cliques (in an appropriate hypergraph $\HH_j$) of the same colour $c$. Moreover, we require that $X$ and $Z$ consist of leaves of $T_j(u)$ and $T_j(v)$, respectively, whereas $Y$ can contain vertices from various leaf sets. To be able to join connectors to each other, we further require that the subtrees of $T_j(u)$ and $T_j(v)$ corresponding to $X$ and $Z$ are isomorphic to the same tree $S$ (here $X$ \emph{corresponds} to the subtree $T'$ if $T'$ is the minimal subtree with leaves $X$; see \Cref{fig:corresponding}). Now let us see how we can combine connectors together. Suppose that the former case from the previous paragraph holds, namely there is a long monochromatic non-grey path in $G^{k_{j+1}}$. Specifically, suppose that $u_1, \ldots, u_n$ are distinct vertices, and $u_i u_{i+1}$ has colour $(c, S)$. Let $X_i, Y_i, Z_i$ be as above for $u_i u_{i+1}$ (so $X_i \subseteq L(T_j(u_i))$ and $Z_i \subseteq L(T_j(u_{i+1}))$, and $X_i$ and $Z_i$ correspond to $S$). Recall that the trees $T_j(u)$ were chosen such that edges of the same `type' have the same colour. This assumption is very useful here: it allows us to assume that $Z_i = X_{i+1}$ for $i \in [n-1]$. It follows that $X_1 \cup Y_1 \cup \ldots \cup Y_{n-1} \cup X_n$ spans a $c$-coloured homomorphic image $\HH'$ of $\HH$. (Actually we obtain a monochromatic homomorphic copy of $K_r(P_n[K_{r}])$, but note that $\HH = P_n^{(r)}$ is a subgraph of this graph.) It is not hard to show, using a property of the sets $Y_i$ that we did not mention, that every vertex in $\HH'$ is the image of few vertices in $\HH$.

	In order to ensure that at some point, namely for some $j \le h$, we find a non-grey monochromatic $P_n$ in the auxiliary colouring, we impose an additional condition on the trees $T_j(u)$ (see \ref{itm:MT}, \Cref{sec:mainproof}). Roughly speaking, we require that there is no $(c, S)$-connector $(X, Y, Z)$ with $X$ and $Z$ corresponding to disjoint copies of $S$ in $T_j(u)$, for any $c$ and $S$. 

	Recall that if there is no non-grey monochromatic $P_n$ in the auxiliary colouring, then there are many disjoint large grey cliques that cover most of the vertices. For a vertex $u$ covered by such a clique $K$, define $T_{j+1}(u)$ to be the tree rooted at $u$ obtained by joining $u$ to the roots of the trees $T_j(v)$ with $v \in K \setminus \{u\}$ (for this we need to first `trim' the trees $T_j(u)$ to ensure the disjointedness of trees within the same clique). The fact that $K$ is a grey clique, which means that there is no $(c, S)$-connector between $T_j(v)$ and $T_j(w)$ with $v, w \in K$, can be used to show that the aforementioned property \ref{itm:MT}  holds for $T_{j+1}(u)$ (assuming that \ref{itm:MT} holds for all trees $T_j(u)$ as well). We show that there is no ordered $d_h$-ary tree of height $h$ which satisfies the above `disconnectedness' property (see \Cref{lem:trees-tight-paths}), for large enough $h$. It thus follows that after at most $h$ steps, we find a non-grey monochromatic $P_n$ in the auxiliary colouring, completing the proof.

\section{Colouring powers of expanders} \label{sec:ramsey-expanders}

	In this section we prove a Ramsey-type results about powers of expanders. Here is the definition of expanders in this context. 

	\begin{defn} \label{def:expander}
		Let $\eps > 0$.
		A graph $G$ is an \emph{$\varepsilon$-expander} if there is an edge between every two disjoint sets of vertices of size at least $\eps|G|$.
	\end{defn}	

	It will be useful to note that there exist bounded degree expanders on any (large) number of vertices, as stated in the following proposition. The proof is standard: take a random graph $G(N, p)$ with suitable parameters $N$ and $p$, and remove large degree vertices; for more details see \Cref{sec:appendix}.

	\begin{prop} \label{prop:existence-expanders}
		For every $\eps > 0$ and sufficiently large $n$, there exists an $\eps$-expander on $n$ vertices with maximum degree at most $80 \cdot (1/\eps) \log(1/\eps)$. 
	\end{prop}
	
	We remind the reader of the definition of $\family_{n, d, \ell}$, which was mentioned in the introduction.
	
	\begin{defn}
	    Given integers $n, d, \ell$, we define $\family_{n, d, \ell}$ to be the family of all graphs on $n$ vertices with maximum degree at most $d$ that can be obtained from the singleton graph by either adding a new vertex and joining it by an edge to an existing one or connecting two existing vertices by a path of length at most $\ell$ whose interior vertices are new.
	\end{defn}

	The main result in this section is the following Ramsey-type lemma about powers of expanders.
	We note that similar results, for $H$ being a path or bounded degree tree, were proved implicitly in previous papers (see, e.g.\ Claim 3.10 in \cite{han2020multicolour} and Claim 22 in \cite{berger2020size}); however, the dependence of the parameters in these previous versions is not good enough for our proof (and the following lemma is applicable to more general families of graphs).

	\begin{lem}\label{lem:ramsey-expanders} 
		Fix $s, t, d, n$, and let $\eps, k, \alpha$ be such that $1 \ll \eps^{-1}$ and $s, t, d \ll k \ll \alpha$. 
		Let $G$ be an $\eps$-expander on $\alpha n$ vertices, and let $H$ be a bipartite graph in $\family_{n,d, 2\log (\alpha n)}$. Then 
		\begin{align*}
			G^k \to \left( K_t,\, \overbrace{H, \ldots, H}^{s} \right).
		\end{align*}
	\end{lem}

	In fact, we will use the following corollary of \Cref{lem:ramsey-expanders}.

	\begin{cor} \label{cor:ramsey-expanders}
		Fix $s, t, d, \mu, n$. Let $\eps, k, \alpha$ be such that $\mu^{-1} \ll \eps^{-1}$ and $s, t, d, \mu^{-1} \ll k \ll \alpha$. 
		Let $G$ be an $\eps$-expander on $\alpha n$ vertices, and let $H$ be a bipartite graph in  $\family_{n,d,2\log(\alpha n)}$. Then for every subset $U \subseteq V(G)$ of size at least $\mu |G|$,
		\begin{align*}
			G[U]^k \to \left( \overbrace{K_t \cup \ldots \cup K_t}^{\text{disjoint, covering $|U|/2$ vertices}}, \overbrace{H, \ldots, H}^{s} \right).
		\end{align*}
	\end{cor}

	\begin{proof}
		Call the first colour grey, and suppose that there is no monochromatic non-grey copy of $H$ in $G^k[U]$. Let $\F$ be a maximal collection of vertex-disjoint grey $K_t$'s. We claim that $\F$ covers at least $|U|/2$ vertices of $U$. Indeed, suppose that this is not the case, and let $W$ be a set of vertices in $U$ not covered by $\F$ of size $(\mu/2)|G| = (\mu \alpha/2)n$. Then $G[W]^k$ contains neither a grey $K_t$ nor a non-grey monochromatic copy of $H$. However, $G[W]$ is a $(2\eps/\mu)$-expander on $(\mu\alpha/2)n$ vertices, a contradiction to \Cref{lem:ramsey-expanders}, applied with $\mu\alpha/2$ and $2\eps/\mu$ instead of $\alpha$ and $\eps$.
	\end{proof}

	\subsection{Preliminaries}

		We shall need two notions of expanders on top of $\eps$-expanders that were defined above. The first is that of a bipartite $\eps$-expander.

		\begin{defn} \label{def:bip-expander}
			Let $\eps > 0$. A bipartite graph $G$ with bipartition $\{X_1, X_2\}$ is a \emph{bipartite $\eps$-expander (with respect to the bipartition $\{X_1, X_2\}$)} if for every two sets of vertices  $A_i \subseteq X_i$ such that $|A_i|\geq \eps |X_i|$ for $i \in [2]$, there is an edge between $A_1$ and $A_2$. 

			A graph $G$ is a \emph{balanced bipartite $\eps$-expander} if it is an $\eps$-bipartite expander with respect to a \emph{balanced bipartition} $\{X_1, X_2\}$, i.e.\ $||X_1|-|X_2||\leq 1$.
		\end{defn}
		
		The next notion is that of $(m, d)$-expanding graphs, which admit a stronger expansion property than $\eps$-expanders for small sets of vertices. For a graph $G$ and a set of vertices $X$ the \emph{neighbourhood $N_G(X)$} of $X$ is the set of vertices in $G \setminus X$ that send an edge to $X$. When $G$ is clear from the context we may omit the subscript $G$. 
		 
		\begin{defn} \label{def:expanding-graphs}
			Let $m, d \ge 1$ be integers. A graph $G$ is \emph{$(m, d)$-expanding} if for every set of vertices $X$ of size at most $m$ we have $|N(X)| \ge d|X|$.
		\end{defn}

		The following lemma and \Cref{cor:expander-boosting} stated below allow us to find large $(m, d)$-expanding subgraphs (for suitable $m, d$) in (bipartite) $\eps$-expanders.

		\begin{lem} \label{lem:bip-expander-boosting}
			Let $d \ge 2$ and $\eps \le 1/8d$.
			Let $G$ be a balanced bipartite $\eps$-expander on $2n$ vertices.
			Then $G$ contains an $(n/4d, d)$-expanding subgraph on at least $n$ vertices.
		\end{lem}

		\begin{proof}
			Denote the bipartition of $G$ by $\{X_1, X_2\}$, so $|X_1| = |X_2| = n$.
			Let $A \subseteq V(G)$ be maximal with $|A| \le n/2d$ and $|N(A)| < d|A|$. We claim that $G' = G \setminus (A \cup N(A))$ is $(n/4d, d)$-expanding.

			Note that since there are no edges from $A \cap X_i$ to $X_{3-i} \setminus (A \cup N(A))$, we have $|A \cap X_i| \le \eps n$ (using that $G$ is an $\eps$-expander and $|A \cup N(A)| < 3n/4$), and so $|A| \le 2\eps n$. It follows that $|G'| \ge 2n - (d+1)2\eps n \ge n$.

			Suppose that there exists $B \subseteq V(G')$ with $|B| \le n/4d$ such that $|N_{G'}(B)| < d|B|$. Then we can add $B$ to $A$ to obtain $|A \cup B| \le n/2d$ (using $\eps \le 1/8d$ and $|A| \le 2\eps n$) and $|N_G(A \cup B)| \le d|A \cup B|$, contradicting the maximality of $A$. Hence $G'$ is an $(n/4d, d)$-expanding graph on at least $n$ vertices, as required.
		\end{proof}	

		\begin{cor}\label{cor:expander-boosting} 
			Let $\eps \le 1/32$.
			Given an $\eps$-expander $G$ on $n$ vertices, there exists an $(n/16, 2)$-expanding subgraph of $G$ on at least $n/2$ vertices.
		\end{cor}
		
		\begin{proof}
			Let $H = G[X_1, X_2]$, where $\{X_1, X_2\}$ is an equipartition of $V(G)$. Then $H$ is a balanced bipartite $2\eps$-expander. The proof follows by applying \Cref{lem:bip-expander-boosting} with $2\eps$, $d = 2$, and $n/2$ instead of $n$.
		\end{proof}

		The following theorem shows that bipartite expanders which are also expanding graphs contain all members of $\family_{n, d, \ell}$, for an appropriate choice of parameters. 
		This is a bipartite variant of a statement that readily follows, e.g., from Section 5.2 in \cite{glebov2013dissertation}. However, as the results in \cite{glebov2013dissertation} do not quite apply to our setting, and for the sake of completeness, we provide a proof of this theorem in \Cref{sec:expanding-thm}.

		\begin{thm} \label{thm:embed-H}
			Let $m, d \ge 3$ be integers. Suppose that $G$ is a $(4m-2, d+2)$-expanding graph, which is bipartite with bipartition $\{X, Y\}$, and where for every two subsets $X' \subseteq X$ and $Y' \subseteq Y$ of size at least $m/8$ there is an edge of $G$ between $X'$ and $Y'$.
			Then $G$ contains every bipartite graph in $\family_{m, d, 2\log m}$.
		\end{thm}
		
		The following lemma follows quite easily from the above results. 

		\begin{lem} \label{lem:subdivision-bip-expander}
			Let $d, n, \ell, N$ be such that $n \le N/32d$ and $\ell \ge \log N$. 
			Then for every bipartite graph $H$ in $\family_{n, d, 2\ell}$ the following holds, where $N' = N/256d$.
			\begin{equation*}
				K_{N, N} \to \left(K_{N', N'}, \,H\right).
			\end{equation*}
		\end{lem}

		\begin{proof}
			Set $\eps = 1/256d$, and consider a red-blue colouring of $K_{N, N}$. We may assume that there is no red $K_{\eps N, \eps N}$, so the graph spanned by the blue edges is a balanced bipartite $\eps$-expander. By \Cref{lem:bip-expander-boosting}, applied with $d+2$, there is an induced subgraph $G'$ of $G$ which is $(N/8d, d+2)$-expanding. By \Cref{thm:embed-H} (applied with $m = N/32d$, using that $G$ is a bipartite $\eps$-expander and that $\ell \ge \log m$), $G'$ contains a copy of $H$, as required.
		\end{proof}

		The following corollary is a multicoloured version of the previous lemma.
		It implies a bipartite version of a recent result of Dragani\'c, Krivelevich and Nenadov \cite{draganic2021rolling} who showed that the $s$-colour size-Ramsey number of `long' subdivisions of bounded degree graphs is linear in the number of vertices of the subdivided graph. 

		\begin{cor} \label{cor:subdivision-bip-expander-multicolour}
			Let $d, n, \ell, s$ be such that $n \le 2^{-(8s-3)}d^{-s} N$ and $\ell \ge \log N$. 
			Then for every bipartite graph $H$ in $\family_{n, d, 2\ell}$ the following holds, where $N' = N/(256d)^s$.
			\begin{equation*}
				K_{N, N} \to \left( K_{N', N'},\, \overbrace{H, \ldots, H}^{s} \right).
			\end{equation*}
		\end{cor}

		\begin{proof}
			Suppose that there is no monochromatic copy of $H$ in one of the last $s$ colours. 
			We show that for every $i \in \{0, \ldots, s\}$ there is a copy of $K_{N_i, N_i}$ that avoids edges of the last $i$ colours. By assumption, this holds for $i = 0$. For $i \in [s]$, suppose that there is a copy of $K_{N_{i-1}, N_{i-1}}$ that avoids the last $i-1$ colours. Apply \Cref{lem:subdivision-bip-expander} (with $N_i$ instead of $N$, using that $n \le N_i/32d$ and $\ell \ge \log N_i$) to find in it a copy of $K_{N_i, N_i}$ that avoids the last $i$-th colour as well. 
			The proof follows from the existence of such a copy for $i = s$.
		\end{proof}

	\subsection{Proof of \Cref{lem:ramsey-expanders}}

		\begin{proof}
			Write $N = |G|$.
			By \Cref{cor:expander-boosting}, there exists an $(N/16, 2)$-expanding subgraph $F$ of $G$ on at least $N/2$ vertices.
			Set $G' = F^{k/t}$.

			\begin{claim} \label{claim:expansion-G'}
				$G'$ is a $2^{-k/3t}$-expander.
			\end{claim}  

			\begin{proof}
				For a subset $X \subseteq V(F)$, let $\Gamma^{\le i}(X)$ be the set of vertices $y$ in $F$ that are at distance at most $i$ (in $F$) from some vertex in $X$.
				Note that $\Gamma^{\le (i+1)}(X) = \Gamma^{\le i}(X) \cup N(\Gamma^{\le i}(X))$ for every $i$. 
				It thus follows from the choice of $F$ that $|\Gamma^{\le (i+1)}(X)| \ge \min\{N/16,\, 2|\Gamma^{\le i}(X)|\}$, implying that 
				\begin{equation*}
					|\Gamma^{\le i}(X)| \ge \min\{N/16, 2^i|X|\}
				\end{equation*}
				for every $X \subseteq V(F)$ and every $i$.

				Let $A, B \subseteq V(F)$ be disjoint subsets of size at least $2^{-k/3t}|F|$. Then 
				\begin{equation*}
					|\Gamma^{\le k/3t}(A)|, |\Gamma^{\le k/3t}(B)| \ge \min\{N/16, |F|\} = N/16 \ge \eps |G|.
				\end{equation*}
				By $\eps$-expansion there is an edge of $G$ between $\Gamma^{\le k/3t}(A)$ and $\Gamma^{\le k/3t}(B)$. This means that there is an edge of $G'$ between $A$ and $B$.
			\end{proof}

			Consider an $(s+1)$-colouring of $G^k$, where the first colour is grey. Our task is to show that there is either a grey $K_t$ or a monochromatic copy of $H$ in a non-grey colour. Suppose that the latter does not hold.

			Define $Y_0 = V(F)$, and recall that $|Y_0| \ge N/2$. We will find sets $X_1, \dots, X_t, Y_1, \dots, Y_t$ such that $X_{i+1}$ and $Y_{i+1}$ are disjoint subsets of $Y_i$ for $i
			\geq 0$; $|X_{i}| = |Y_{i}| = 2^{-(i+1)} (256d)^{-si} N$; and there are no non-grey edges of $G^k$ between $X_i$ and $Y_i$. Suppose that $X_1, \ldots, X_i, Y_0, \ldots, Y_i$ are defined for some $i \in [t]$. Consider an arbitrary partition $\{Y_i', Y_i''\}$ of $Y_i$ into two parts of size $|Y_i|/2$, and apply \Cref{cor:subdivision-bip-expander-multicolour} to $G^k[Y_i', Y_i'']$ (with $\ell = \log(\alpha n)$ and $|Y_i|/2$ instead of $N$) to obtain disjoint sets $X_{i+1}$, $Y_{i+1}\subseteq Y_{i}$ with $|X_{i+1}|=|Y_{i+1}| = (256d)^{-s} |Y_i|/2 = 2^{-(i+2)}(256d)^{-s(i+1)}N$ such that all edges in $G^k[X_{i+1}, Y_{i+1}]$ are grey, using that there is no monochromatic copy of $H$ in a non-grey colour. 
			The corollary is applicable since $s, t, d \ll \alpha = N/n$, and $\log (\alpha n) \ge \log(|Y_i|/2)$, for $i \in \{0,\dots, t-1\}$.

			Note that, as $s, t, d \ll k$, we have $|X_i| \ge 2 \cdot 2^{-k/3t} |F|$, so by \Cref{claim:expansion-G'}, there is an edge of $G'$ between every two subsets $X_i' \subseteq X_i$ and $X_j' \subseteq X_j$ with $|X_i'| \ge |X_i|/2$ and $|X_j'| \ge |X_j|/2$. 
			Let $Z_i$ be the set of vertices $z_i \in X_i$ such that there exists a path $z_1 \ldots z_i$ in $G'$ with $z_j \in X_j$ for $j \in [i]$. One can show, by induction, that $|Z_i| \ge |X_i|/2$ for all $i \in [t]$. In particular, $Z_t \neq \emptyset$, so there is a path $z_1 \ldots z_t$ in $G' \subseteq G^{k/t}$ with $z_i \in X_i$ for $i \in [t]$.
			Then $\{z_1, \ldots, z_t\}$ is a grey clique in $G^k$.
		\end{proof}

		\begin{rem}
			Some effort can be spared if instead of proving \Cref{thm:main} in full generality, one settles for a special case.

			Indeed, it is much easier to prove a variant of \Cref{thm:embed-H} for $H = P_n$, e.g., using a consequence of the Depth First Search algorithm (see, e.g., Lemma 2.3 in \cite{ben2012long}). This suffices to prove the special case of our main result, \Cref{thm:main-path}, for powers of tight paths.

			Similarly, a well-known result of Friedman and Pippenger \cite{friedman1987expanding} (whose arguments we use in the proof of \Cref{thm:embed-H}) asserts that $(2m, d+1)$-expanding graphs contain every tree in $\T_{m, d}$, which readily implies the special case of \Cref{thm:embed-H} where $H \in \T_{m, d}$. This suffices to prove the version of our main result, \Cref{thm:main-tree}, for powers of trees.
		\end{rem}

\section{Colouring ordered trees} \label{sec:colouredorderedtrees}

	We will use the notion of ordered trees and forests, introduced in the following two definitions.

	\begin{defn}[ordered trees]\label{def:ordered-tree}
		Let $T$ be a rooted tree. Denote its non-root leaves by $L(T)$, and for any vertex $u$ in $T$ denote the subtree of $T$ rooted at $u$ by $T_u$. We say that $T$ has \emph{height} $h$ if all non-root leaves of $T$ are at distance exactly $h$ from the root.

		A \emph{$d$-ary tree of height $h$} is a tree of height $h$, whose vertices that are not leaves have $d$ children.

		An \emph{ordered tree} of height $h$ is a rooted tree $T$ of height $h$, along with an ordering $\sigma$ of its leaves that corresponds to a planar drawing of $T$. More precisely, there exists a sequence $\sigma_1, \ldots, \sigma_h = \sigma$, such that $\sigma_i$ is an ordering of the vertices $L_i$ whose distance from the root is $i$, and the following holds for $i \in [h-1]$:
		for every $p, q \in L_i$, if $p$ precedes $q$ in $\sigma_i$, then all children of $p$ precede all children of $q$ in $\sigma_{i+1}$.
	\end{defn}
	
	\begin{figure}[h]
	    \centering
	    \includegraphics[scale = 1]{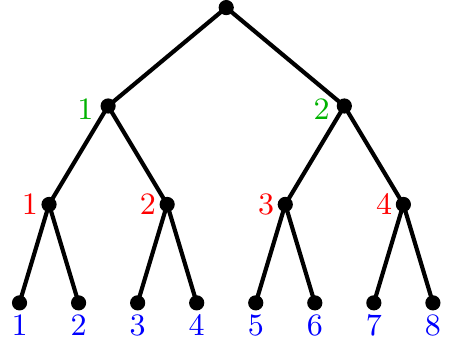}
	    \hspace{1cm}
	    \includegraphics[scale = 1]{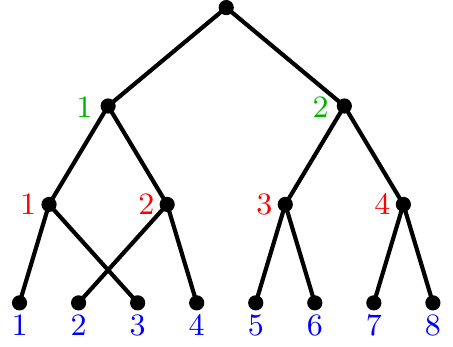}
	    \vspace{-.3cm}
	    \caption{An example of an ordered tree (on the left) and an `unordered' tree (on the right). Here the numbering of the vertices in level $i$ corresponds to the ordering $\sigma_i$.}
		\label{fig:ordered-tree}
	\end{figure}

	\begin{defn}[ordered forests] \label{def:ordered-forest}
		A \emph{rooted forest} is a forest whose components are rooted trees. Given a forest $F$, denote by $L(F)$ the set of non-root leaves of $F$. We say that $F$ has height $h$ if each of its components has height $h$.

		A \emph{$d$-ary forest of height $h$} is a rooted forest whose components are $d$-ary trees of height $h$. 

		An \emph{ordered forest} is a rooted forest whose components are ordered trees, along with an ordering of its components.
	\end{defn}

	Given an ordered tree $T$ we often use the natural correspondence between subsets of $L(T)$ and ordered subtrees of $T$. For $X \subseteq L(T)$, we say that \emph{$X$ corresponds to $S \subseteq T$} if $S$ is the minimal subtree of $T$ that contains $X$. (See \Cref{fig:corresponding} for a set $X$ and the corresponding subtree.)
	
	\begin{figure}[h]
	    \centering
	    \includegraphics[scale = 1]{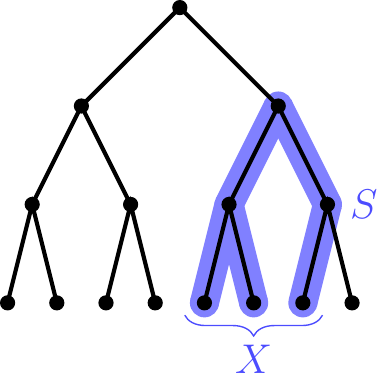}
	    \vspace{-.3cm}
	    \caption{A set $X$ of leaves and the subtree $S$ corresponding to $X$.}
		\label{fig:corresponding}
	\end{figure}
	
	We use a similar correspondence between subsets of leaves of an ordered forest $F$ and subforests of $F$. More precisely, given an ordered forest $F$ with components $T_1, \ldots, T_k$ (in this order), a subset $X$ of $L(T)$ \emph{corresponds to} an ordered subforest $S$ of $F$ with components $S_1 \subseteq T_1, \ldots, S_k \subseteq T_k$, if $S_i$ is the subtree corresponding to $X \cap L(T_i)$, for $i \in [k]$. 
	
	In this section we prove two Ramsey-type results about ordered trees and forests.
	The next lemma, whose proof we delay to the next subsection, is a Ramsey-type result about $D$-ary ordered forests whose copies of a fixed forest $S$ are $s$-coloured.

	\begin{lem} \label{lem:ramsey-trees}
		Let $d,r, h, k,s \ll D$.
		Let $F$ be a $D$-ary ordered forest of height $h$ with $k$ components. Let $S$ be an ordered forest of height $h$ with $r$  leaves and $k$ components. Given any $s$-colouring of the copies of $S$ in $F$, there is an ordered $d$-ary subforest $F'$ of $F$ of height $h$ with $k$ components such that all copies of $S$ in $F'$ have the same colour.
	\end{lem}

	We obtain the following corollary by repeatedly applying \Cref{lem:ramsey-trees}. The notation in this corollary is somewhat different than what we used in the lemma; this is to make it easier to apply the corollary in \Cref{sec:hanging-trees}.

	\begin{cor} \label{cor:ramsey-trees}
		Let $d,r, h, \rho, s \ll D$.
		Let $T_1, \ldots, T_{\rho}$ be vertex-disjoint $D$-ary ordered trees of height $h$, and let $\HH$ be the $r$-uniform complete graph on $L(T_1) \cup \ldots \cup L(T_{\rho})$.
		For any $s$-colouring of $\HH$, there exist $d$-ary subtrees $T_i' \subseteq T_i$ of height $h$, for $i \in [\rho]$, such that for any sequence $S = (S_1, \ldots, S_{\rho})$ of ordered trees of height at most $h$ and $r$ leaves in total, the following holds: all edges of form $L(S')$, where $S' = (S_1', \ldots, S_{\rho}')$ is a copy of $S'$ (in order, namely $S_i'$ is a copy of $S_i$) with $S_i' \subseteq T_i'$ and $L(S_i') \subseteq L(T_i')$, have the same colour.
	\end{cor}

	\begin{proof} 
		Note that it suffices to consider sequences $S = (S_1, \ldots, S_{\rho})$ such that in each $S_i$, all non-root leaves are at the same distance from the root, because for sequences $S$ that violate this, there are no copies $(S_1', \ldots, S_{\rho}')$ of $S$ in $(T_1, \ldots, T_{\rho})$ with $S_i' \subseteq L(T_i)$ for $i \in [\rho]$. 
		Denote by $\ell$ the number of such sequences $S$ and observe that $\ell$ is a function of $r$, $h$ and $k$.
		Thus, by choice of the parameters, there exists a sequence $d_0, \ldots, d_{\ell}$ such that $d_0 = D$, $d_{\ell} = d$, and $r, h, k, s, d_{\ell} \ll d_{\ell-1} \ll \ldots \ll d_0$. 

		Fix an ordering $(S^j)_{j \in [\ell]}$ of all sequences as in the statement. 
		We claim that there exist sequences $T_i = T_{i,0} \supseteq \ldots \supseteq T_{i, \ell}$, for $i \in [\rho]$, such that: $T_{i, j}$ is a $d_j$-ary ordered tree of height $h$; and all edges of $\HH$ of form $L(S')$, where $S' = (S_1', \ldots, S_{\rho'})$ is a copy of $S^j$ (in order) with $S_i' \subseteq T_{i, j}$ and $L(S_i') \subseteq L(T_{i, j})$, have the same colour. 
		
		To see this, suppose that for some $j \in [\rho]$, the trees $T_{1, j-1}, \ldots, T_{\rho, j-1}$ satisfy the properties listed in the previous paragraph. We now show how to obtain suitable trees $T_{1, j}, \ldots, T_{\rho, j}$. Consider the sequence of trees $S^j = (S_1, \ldots, S_{\rho})$. Let $k$ be the number of indices $i$ for which $S_i$ is non-empty, and for convenience assume that $S_1, \ldots, S_k$ are non-empty. Replace each tree $S_i$ by an ordered tree $S_i'$ of height $h$, obtained by adding a path of appropriate length to the root of $S_i$ and making the other end of the path be the root of $S_i'$ (here we use the assumption that in each $S_i$, all non-root leaves are at the same distance from the root). 

		Let $S$ be the ordered forest with components $S_1', \ldots, S_k'$ (in this order), and let $F$ be the forest with components $T_{1, j-1}, \ldots, T_{k, j-1}$. Consider the $s$-colouring of copies $S'$ of $S$ in $F$ with $L(S') \subseteq L(F)$ according to the colour of $L(S')$ in $\HH$. Apply \Cref{lem:ramsey-trees} with the forests $F$ and $S$, to obtain a $d_j$-ary forest $(T_{1, j}, \ldots, T_{k, j})$ of height $h$ whose copies of $S$ all have the same colour. For $i \in [k+1, \rho]$, take $T_{i, j}$ to be any $d_j$-ary subtree of $T_{i, j-1}$ of height $h$. The trees $T_{1, j}, \ldots, T_{\rho, j}$ satisfy the desired properties.
		
		To complete the proof of the corollary, take $T_i' = T_{i, \rho}$. 
	\end{proof}

	The next lemma shows that for every tall enough ordered $d$-ary tree, and every $s$-colouring of $r$-sets of leaves, there exists a monochromatic structure that `connects' the leaves of two isomorphic subtrees with $\ell$ leaves. 

	\begin{lem} \label{lem:trees-tight-paths}
	    Let $r \le \ell, s \ll h \ll d$.
		Let $T$ be a $d$-ary ordered tree of height $h$, and let $\HH$ be the $r$-uniform complete graph on $L(T)$.
		Then for every $s$-colouring of $\HH$ there exist disjoint sets $X, Y, Z \subseteq L(T)$ of size $\ell$, such that the  subtrees $S_x, S_z \subseteq T$ that correspond to  $X$ and $Z$ are isomorphic and vertex-disjoint, and $X \cup Y$ and $Y \cup Z$ are monochromatic cliques of the same colour.
	\end{lem}

	\begin{proof}
		Fix an $s$-colouring of $r$-subsets of $L(T)$.
		By \Cref{cor:ramsey-trees} applied with $k=1$, $d = 2$ and $F = T$, there is a binary subtree $T' \subseteq T$ of height $h$, such that $r$-sets of leaves that correspond to isomorphic ordered subtrees of $T$ have the same colour. 
		Let $v_h$ be the first vertex in $L(T')$ (according to the ordering of $L(T')$ that $T'$ is equipped with), and let $v_0 \ldots v_h$ be the path from the path from the root of $T'$ to $v_h$. 
		Let $u_i$ be the last leaf in the tree rooted at $v_i$ for $i \in \{0, \ldots, h\}$. (See \Cref{fig:ramsey-tree} for an illustration of the vertices $v_i$ and $u_i$.)
		By Ramsey's theorem there is a subset $A \subseteq \{u_0, \ldots, u_h\}$ of size $2\ell+1$ whose $r$-subsets all have the same colour, say red. Let $w_1, \ldots, w_{2\ell + 1}$ be the vertices in $A$, in order. 
		\begin{figure}[h]
			\centering
			\includegraphics[scale = .9]{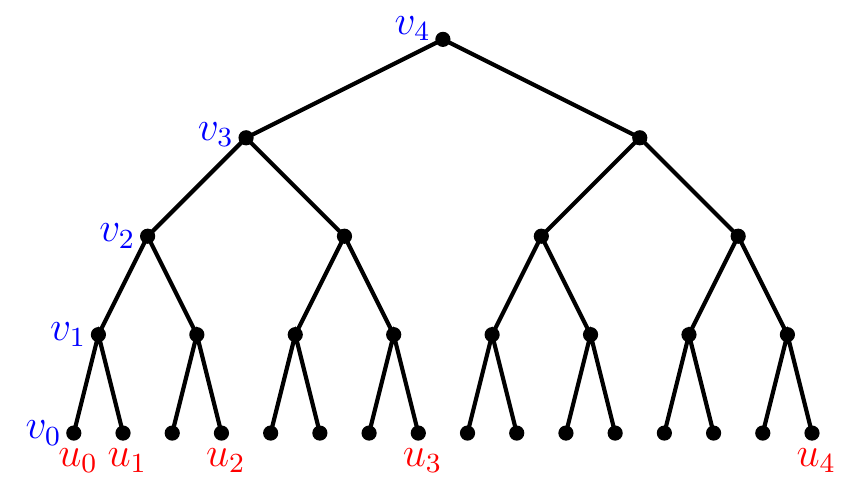}
			\hspace{.5cm}
			\includegraphics[scale = .9]{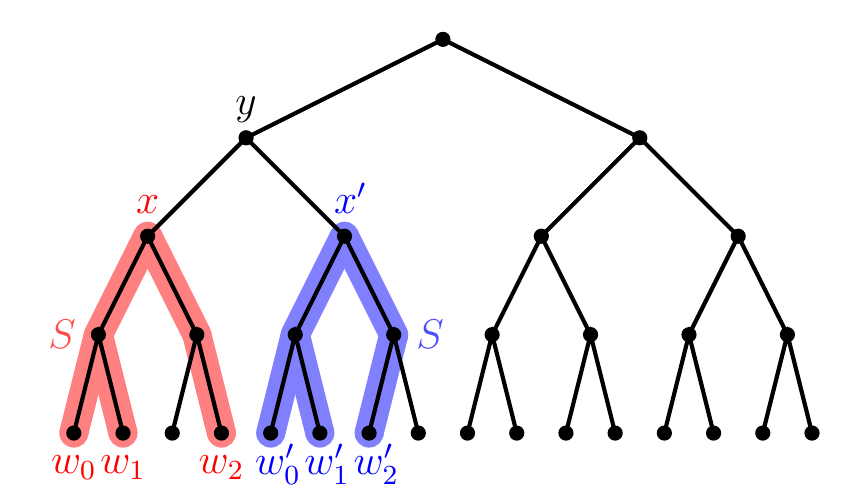}
			\vspace{-.3cm}
			\caption{An illustration of the sets $\{v_0, \ldots, v_h\}$, $\{u_0, \ldots, u_h\}$, $\{w_1, \ldots, w_{\ell}\}$ and $\{w_1', \ldots, w_h'\}$.}
			\label{fig:ramsey-tree}
		\end{figure}

		Let $S$ be the minimal ordered subtree of $T'$ whose leaves are $\{w_1, \ldots, w_{\ell}\}$. Denote its root by $x$ and denote the father of $x$ by $y$. Let $x'$ be any child of $y$ other than $x$, let $S'$ be a copy of $S$ rooted at $x'$, and denote the leaves of $S'$ by $\{w_1', \ldots, w_{\ell}'\}$. (See \Cref{fig:ramsey-tree} for an illustartion of the vertices $w_i$ and $w_i'$.) 

		Note that $\{w_1, \ldots, w_{\ell}, w_{\ell+2}, \ldots, w_{2\ell + 1}\}$ and $\{w_1', \ldots, w_{\ell}', w_{\ell+2}, \ldots, w_{2\ell + 1}\}$ correspond to isomorphic trees (we omitted $w_{\ell+1}$ as it may be a descendant of $y$, whereas $w_{\ell+1}, \ldots, w_{2\ell+1}$ are not), so by choice of $w_1, \ldots, w_{2\ell + 1}$ and $T'$ both sets are red cliques. Take $X = \{w_1, \ldots, w_{\ell}\}$, $Y = \{w_{\ell+2}, \ldots, w_{2\ell + 1}\}$ and $Z = \{w_1', \ldots, w_{\ell}'\}$. These sets satisfy the requirements of the lemma.
	\end{proof}

	It remains to prove \Cref{lem:ramsey-trees}, which we do in the following subsection.
	Both proofs will use Ramsey's theorem, stated here.
	($K_t^{(r)}$ is the complete $r$-uniform hypergraph on $t$ vertices.)

	\subsection{Proof of Lemma~\ref{lem:ramsey-trees}}

		\begin{proof} 
			We prove the lemma by induction on $h$ and $k$. 
			The initial case is $h=k= 1$, namely $F$ and $S$ are both stars, with $D$ and $r$ leaves, respectively, and so copies of $S$ correspond to $r$-subsets of leaves of $F$. This case thus follows from Ramsey's theorem.

			Fix $(h,k)\neq(1,1)$ and suppose that the lemma holds for all $(h',k')$ with either $h'=h, k'<k$ or with $h'<h, k' \geq 1$. Let $\chi$ be an $s$-colouring of the copies of $S$ in $F$.

			Suppose first that $k = 1$ and $h \ge 2$, so $F$ and $S$ are trees of height $h$.  Denote the root of $F$ by $r_F$ and the root of $S$ by $r_S$.
			Let $F' = F \setminus \{r_F\}$, let $S' = S \setminus \{r_S\}$, and let $\ell$ be the number of components in $S'$; so $1 \le \ell \le r$ as $S$ has $r$ leaves. Consider the $s$-colouring $\chi'$ of copies of $S'$ in $F'$ defined as follows: if $S''$ is a copy of $S'$ in $F'$, consider the copy of $S$ obtained by joining $r_F$ to the roots of the components of $S'$ (since both $S'$ and $F'$ have height $h-1$, the roots of components of $S'$ are roots of components of $F'$, so the copy of $S$ formed in this way is indeed a subtree of $F$), and colour $S''$ by the colour of this copy of $S$ in $F$ according to $\chi$. 

			Let $\alpha$ be such that $d, r, h, s \ll \alpha \ll D$, and set $t = \binom{\alpha}{\ell}$. 
			By choice of parameters, there is a sequence $d_0, \ldots, d_t$ such that $d_0 = D$, $d_t = d$, and $r, h, s, d_t \ll d_{t-1} \ll \ldots \ll d_0$.
			Let $A$ be a set of $\alpha$ children of $r_F$ in $F$, and let $A_1, \ldots, A_{t}$ be any enumeration of its $\ell$-subsets. 
			Let $T_a$ be the subtree of $F$ rooted at $a$.
			We claim that there exist sequences $T_a = T_a^{(0)} \supseteq T_a^{(1)} \supseteq \ldots \supseteq T_a^{(t)}$ for $a \in A$, where $T_a^{(i)}$ is a $d_i$-ary ordered tree of height $h-1$ with the following property: all copies of $S'$ in the forest $\bigcup_{a \in A_i} T_a^{(i)}$ have the same colour. To see this, use apply the induction hypothesis with $h' = h-1$ and $k' = \ell$, to $\bigcup_{a \in A_i} T_a^{(i-1)}$, letting $\bigcup_{a \in A_i} T_a^{(i)}$ be the resulting subforest, and take $T_a^{(i)} \subseteq T_a^{(i-1)}$ to be an arbitrary $d_i$-ary ordered subforest of height $h-1$ for $a \in A \setminus A_i$. The subtrees $T_a^{(i)}$ satisfy the requirements. Denote by $c_i$ the colour of any copy of $S'$ in $\bigcup_{a \in A_i} T_a^{(i)}$.

			Consider the auxiliary colouring of the complete $\ell$-graph on $A$, where the edge $A_i$ has colour $c_i$. Then by Ramsey's theorem and choice of $\alpha$ there is a monochromatic subset $A' \subseteq A$ of size $d$; say the common colour is $c$. Let $F''$ be the $d$-ary tree of height $h$ obtained by reattaching the root of $F$ to the subtrees $T_a^{(t)}$ with $a \in A'$. Then $F'' \subseteq F$ is an ordered $d$-ary tree of height $h$ whose copies of $S$ all have colour $c$, as required.

			Now suppose $k \ge 2$.
			Let $T$ be the first tree in $F$, let $F' = F \setminus T$, let $R$ be the first tree in $S$ and let $S' = S \setminus R$. Choose $d'$ satisfying $d, r, h, k, s \ll d' \ll D$, and let $T'$ be any $d'$-ary subtree of $T$ of height $h$. Enumerate the copies of $R$ in $T'$ by $R_1, \ldots, R_t$. By choice of $d'$ and $t$ there exist $d_0, \ldots, d_t$ such that $d_0 = D$, $d_t = d$, and $r, s, h, k, d_t \ll d_{t-1} \ll \ldots \ll d_0$. 
			By induction, there exist subforests $F' \supseteq F_1 \supseteq \ldots \supseteq F_t$, such that $F_i$ is a $d_i$-ary forest of height $h$ with $k-1$ components, whose copies of $S$ in $T' \cup F_i$ that contain $R_i$ all have the same colour, denoted by $c_i$.
			Now consider the colouring of copies of $R$ in $T'$, where $R_i$ is coloured $c_i$. By induction, $T'$ contains a $d$-ary subtree $T''$ of height $h$ whose copies of $R$ all have the same colour, say red. Then the forest $T'' \cup F_t$ is a $d$-ary forest of height $h$ with $k$ components whose copies of $S$ are red.
		\end{proof}

\section{Hypergraphs with hanging trees} \label{sec:hanging-trees}

	In this section we prove auxiliary results that will help us in the proof of the main result, given in \Cref{sec:main-proof}. In most of them we are given a bounded degree hypergraph $\G$, where for each vertex $u$ there is a $D$-ary ordered tree $T(u)$ of height $h$, rooted at $u$, for a large constant $D$. The results then tell us that we can `trim' each tree, namely replace it by a $d$-ary subtree $T'(u)$ of height $h$, with useful properties, provided that $d\ll D$.

	The following easy lemma allows us to assume trim two possibly intersecting trees so that the resulting trees are vertex-disjoint.

	\begin{lem} \label{lem:untangling-trees}
		Let $d, h \ll D$.
		Let $T_1, T_2$ be $D$-ary ordered trees of height $h$ with distinct roots. Then there exist vertex-disjoint $d$-ary subtrees of height $h$, denoted by $T_1', T_2'$, such that $T_1' \subseteq T_1$, $T_2' \subseteq T_2$. 
	\end{lem}

	\begin{proof} 
		Assign to each non-root vertex $u$ of $V(T_1) \cup V(T_2)$ a random variable $\tau(u)$ which is chosen to be either $1$ or $2$ uniformly at random, independently of other elements. We claim that with positive probability, for every non-leaf in $T_i$, at least $D/3$ of its children are assigned $i$, for $i \in [2]$. Indeed, the probability that this fails for a particular non-leaf in either $T_1$ or $T_2$ is at most $e^{-D/36}$, by Chernoff's bounds, and thus by a union bound the probability that some non-leaf in $T_i$ has fewer than $D/3$ children that were assigned $i$, for $i \in [2]$, is at most $2D^h e^{-D/36} < 1$ (using $D \gg h$), as claimed. It follows that there exist subtrees $T_i' \subseteq T_i$, for $i \in [2]$, such that $T_i'$ is a $d$-ary tree of height $h$ and its non-root vertices were assigned $i$ (and are not the root of $T_{3-i}$; using $D \gg d$). The trees $T_1', T_2'$ are vertex-disjoint (using that the roots of $T_1$ and $T_2$ are distinct), as required.
	\end{proof}

	The next lemma follows from the previous one by applying it to the setting mentioned at the beginning of the section.

	\begin{lem} \label{lem:trimming-order}
		Let $d, h, r, \Delta \ll D$. Let $G$ be a graph with maximum degree $\Delta$. Given a collection of trees $\{T(u)\,|\,u\in V(G)\}$, which are $D$-ary ordered trees of height $h$, there exists a collection of $d$-ary subtrees $\{T'(u) \subseteq T(u)\,|\,u\in V(G)\}$ of height $h$ such that for every edge $uv$ in $G$ the trees $T'(u)$ and $T'(v)$ are vertex-disjoint. 
	\end{lem}

	\begin{proof}
		Let $d_0, \ldots, d_{\Delta + 1}$ be such that $d_0 = D$, $d_{\Delta + 1} = d$, and $h, r, \Delta, d_{\Delta+1} \ll d_{\Delta} \ll \ldots \ll d_0$.
		
		By Vizing's theorem, there is a proper colouring of the edges of $G$ by $\Delta+1$ colours. In other words, there is a collection of pairwise edge-disjoint matchings $M_1, \ldots, M_{\Delta+1}$ that covers the edges of $G$. 
		We claim that for every vertex $u$ there exists a sequence $T(u) = T_0(u) \supseteq T_1(u) \ldots \supseteq T_{\Delta+1}(u)$ such that $T_i(u)$ is an ordered $d_i$-ary tree of height $h$; and $T_i(u)$ and $T_i(v)$ are vertex-disjoint for every $uv \in M_i$.
		To see this, given trees $T_{i-1}(u)$ as above, apply \Cref{lem:untangling-trees} to the trees $T_{i-1}(u)$ and $T_{i-1}(v)$ for every edge $uv$ in $M_i$, thus yielding trees $T_i(u)$ and $T_i(v)$ that satisfies the requirements. For any vertex $u$ which is not contained in some edge in $M_i$, let $T_i(u)$ be an arbitrary $d_i$-ary subtree of $T_{i-1}(u)$ of height $h$. 
		Take $T'(u) = T_{\Delta + 1}(u)$ for $u \in V(G)$.
	\end{proof}

	In the proof of our main result, we will consider $s$-colourings of $r$-graphs of the form $K_r(G^k[K_B])$ (whose edges are the $r$-cliques in the complete blow-up of $G^k$ where each vertex is replace by a $B$-clique), where $G$ is an expander with bounded degree. The following lemma allows us to assume, by switching to a suitable copy of $K_r(G^k[K_b])$ for some $b\ll B$, that all edges of the same `type' (namely, they have the same intersection size with the blob corresponding to $u$, for every vertex $u$ in $G$) have the same colour.

	\begin{lem} \label{lem:monochromatic-blowup}
		Let $\Delta, b, r, s \ll B$.
		Let $G$ be a graph with maximum degree $\Delta$. Consider the hypergraph $\HH$, with vertices $V(\HH) = \bigcup_{u \in V(G)} B(u)$, where $B(u)$ is a set of $B$ elements and the sets $B(u)$ are pairwise disjoint, and edges 
		\begin{equation*}
			E(\HH) = \{(v_1, \ldots, v_r) : v_i \in B(u_i) \text{ and $\{u_1, \ldots, u_r\}$ is a clique in $G$}\}.
		\end{equation*}
		Then for every $s$-colouring of $\HH$, there exist subsets $B'(u) \subseteq B(u)$ of size $b$, such that in the subhypergraph $\HH'$ of $\HH$, induced by $\bigcup_{u \in V(G)} B'(u)$, every two edges $e$ and $f$ in $\HH'$, with $|e \cap B'(u)| = |f \cap B'(u)|$ for every $u \in V(G)$, have the same colour.
	\end{lem}

	\begin{proof}
		This follows from \Cref{lem:trimming-ramsey} below, applied with $h = 1$ (by letting $T(u)$ be a star with leaves $B(u)$ and root $u$, for $u \in V(G)$, and assuming that the sets $B(u)$ do not intersect $V(G)$).
	\end{proof}			

	In fact, we will need a stronger version of \Cref{lem:monochromatic-blowup}, see \Cref{lem:trimming-ramsey} below, which is applicable for graphs with ordered trees hanging from each vertex. In order to prove the latter lemma, we first prove the following lemma.
	It says that given a hypergraph $\G$ with bounded degree and bounded edge size, and a collection of ordered $D$-ary trees $T(u)$ of height $h$ hanging from each vertex $u$, the following holds. Given any $s$-colouring of $r$-sets of leaves whose roots form an edge of $\G$, the trees $T(u)$ can be `trimmed' so that in the resulting structure any two $r$-sets of leaves of the same `type' have the same colour.

	\begin{lem} \label{lem:trimming-ramsey}
		Fix $d, r, h, \Delta \ll D$. Let $\G$ be a hypergraph with  $\Delta(\G)\leq \Delta$ where every edge has size at most $r$. Suppose we are given a collection of $D$-ary trees $\{T(u) \,|\,u\in V(\G)\}$ of height $h$ such that trees corresponding to vertices of an edge of $\G$ are vertex-disjoint. Let $\HH$ be the $r$-uniform hypergraph on vertices $\bigcup_{u}L(T(u))$ whose edges are as follows: for every edge $e$ in $\G$, any $r$ elements in $\bigcup_{u \in e}{L(T(u))}$ form an edge.  
		
		Then for every $s$-colouring of $\HH$, there exists a collection of  $d$-ary subtrees $\{T'(u)\,|\,T'(u)\subseteq T(u)\}$ of height $h$ such that  the following holds for all $\rho\leq r$. 
        For any sequence of ordered trees $S = (S_1, \dots, S_{\rho})$  of height at most $h$ and with $r$ leaves in total the following holds. For every $(u_1, \dots u_{\rho}) \in E(\G)$  there is a colour $c$ such that for all copies $S' = (S_1', \dots, S_{\rho}')$  of $S$ (in order, namely $S_i'$ is a copy of $S_i$) such that $S_i'\subseteq T'(u_i) $ and $L(S_i')\subseteq L( T'(u_i) ) $, the edge $L(S')$ has colour $c$.
	\end{lem}

	\begin{proof}
		Let $m = r \Delta + 1$, and let $d_0, \ldots, d_m$ be such that $d_0 = D$, $d_m = d$ and $r, h, \Delta, d_m \ll d_{m-1} \ll \ldots \ll d_0$.

		Consider the line graph $\LG(\G)$ of $\G$, and observe that it has maximum degree at most $r \Delta = m - 1$, implying that there is a proper colouring of the vertices of $\LG(\G)$ with $m$ colours.
		In other words, there exist pairwise disjoint families $\F_1, \ldots, \F_m$ of matchings (i.e.\ pairwise vertex-disjoint edges) that cover all edges in $\G$.

		We claim that for every vertex $u$ in $\G$, there exist subtrees $T(u) = T_0(u) \supseteq \ldots \supseteq T_m(u)$ as follows: $T_i(u)$ is a $d_i$-ary tree of height $h$; for every edge $(u_1, \ldots, u_{\rho})$ in $\F_i$, and for every sequence of trees $S = (S_1, \ldots, S_{\rho})$ as in the statement of the lemma, all edges of $\HH$ of form $L(S')$, where $S' = (S_1', \ldots, S_{\rho'})$ is a copy of $S$ (in order) with $S_j' \subseteq T_i(u_j)$ and $L(S_j') \subseteq L(T_i(u_j))$, have the same colour.
		To see this, given trees $T_{i-1}(u)$ with these properties, apply \Cref{cor:ramsey-trees} to each edge $e$ in $\F_i$ to obtain trees $T_i(u)$ for $u \in e$ with the required properties; for any vertex $u$ which is not contained in an edge in $\F_i$, let $T_i(u)$ be an arbitrary $d_i$-ary subtree of $T_{i-1}(u)$ of height $h$. 
		Set $T'(u) = T_m(u)$.
	\end{proof}

\section{Proof of the main result} \label{sec:main-proof}

	In this section we prove our main result, \Cref{thm:main}. To get started, we recall some notation.
	For a graph $G$, denote by $G^-$ its $1$-subdivision (the graph obtained by replacing each edge of $G$ by a path of length $2$, whose middle vertex is new), and recall that $K_r(G)$ is the $r$-graph whose edges are $r$-cliques in $G$, and $G[K_r]$ is the blow-up of $G$ where every vertex is replaced by a $r$-clique. Recall also that $\family_{n, d, \ell}$ is the family of $n$-vertex graphs with maximum degree $d$ that can be obtained from a singleton by successively either adding a vertex of degree $1$ or connecting two vertices by a path of length at least $\ell$ whose interior is new. 

	The main ingredient in the proof of \Cref{thm:main} is the following result, which allows us to find a monochromatic copy of $\HH$ in a hypergraph $\G$ obtained by modifying a bounded degree expander, for every $\HH$ of the form $\HH = K_r(Q^-[K_t])$, where $Q$ is a bipartite graph in $\family_{n, D, \ell}$ with $\ell = \Omega(\log n)$.

	\begin{thm} \label{thm:main-auxiliary}
		Let $r \le t, s, D \ll \eps^{-1}, \Delta \ll k \ll B \ll \alpha$. The following holds for every integer $n$. Let $Q$ be a bipartite graph in $\family_{n, D, 2\log(\alpha n)}$, and let $G$ be an $\eps$-expander on $\alpha n$ vertices with maximum degree at most $\Delta$.
		Then in every $s$-colouring of $K_r(G^{k}[K_B])$ there is a monochromatic copy of $K_r(Q^-[K_{t}])$.
	\end{thm}

	\Cref{thm:main-auxiliary} follows quite easily from the following weaker version of \Cref{thm:main-auxiliary}, where instead of finding a monochromatic copy of $\HH$, we find a monochromatic copy of a homomorphic image $\HH'$ of $\HH$ where every vertex of $\HH'$ is the image of few vertices in $\HH$, using a result from the previous section. (Recall that a map $\phi : \HH \to \HH'$ is a \emph{homomorphism} if edges of $\HH$ are mapped to edges of $\HH'$.)
	The proof of \Cref{thm:main-homomorphic} appears in the next subsection.

	\begin{thm} \label{thm:main-homomorphic}
		Let $r \le t, s, D \ll \eps^{-1}, \Delta \ll k  \ll \alpha$. The following holds for every integer $n$. Let $Q$ be a bipartite graph in $\family_{n, D, 2\log(\alpha n)}$, and let $G$ be an $\eps$-expander on $\alpha n$ vertices with maximum degree at most $\Delta$.
		Then in every $s$-colouring of $K_r(G^{k})$ there is a monochromatic homomorphic image $\HH'$ of $\HH = K_r(Q^-[K_{t}])$ such that every vertex of $\HH'$ is the image of at most $(D+1) \Delta^{(2k + 1)}$ vertices of $\HH$.
	\end{thm}

	\begin{proof}[Proof of \Cref{thm:main-auxiliary} using \Cref{thm:main-homomorphic}]
		Let $b= (D+1) \Delta^{(2k + 1)}$, and note that by our choice of $D,\Delta,k$ we can ensure that $b \ll B$.

		Let $\G$ be the hypergraph on vertices $\bigcup_{u \in V(G)} B(u)$, where $B(u)$ is a set of $B$ vertices and the sets $B(u)$ are pairwise disjoint, and edges 
		\begin{equation*}
			E(\G) = \{(v_1, \ldots, v_r) \,:\, \exists \, u_1, \ldots, u_r \text{ s.t.\ $v_i \in B(u_i)$ and $\{u_1, \ldots, u_r\}$ is a clique in $G^k$}\}.
		\end{equation*}
		(The vertices $u_1, \ldots, u_r$ need not be distinct.) So $\G \cong K_r(G^{k}[K_B])$.
		We will show that $\G \sarrow \HH$.

		Fix an $s$-colouring of $\G$.
		By \Cref{lem:monochromatic-blowup} there exist subsets $B'(u) \subseteq B(u)$ of size $b$, such that in the subhypergraph $\G'$ of $\G$, induced on $\bigcup_{u \in V(G)} B'(u)$, if edges $e$ and $f$ satisfy $|e \cap B'(u)| = |f \cap B'(u)|$ for every $u \in V(G)$, then $e$ and $f$ have the same colour.
		Consider the hypergraph $\G'' = K_r(G^k)$, along with the $s$-colouring inherited from $\G'$ as follows: given an edge $(u_1, \ldots, u_r)$ in $\G''$ colour it according to the colour of $(v_1, \ldots, v_r)$ in $\G'$, for any choice of vertices $v_i \in B'(u_i)$ for $i \in [r]$ (by assumption on $\G'$ this is well-defined; namely, the colour does not depend on the choice of $v_1, \ldots, v_r$). 

		Now, by \Cref{thm:main-homomorphic}, there is a monochromatic, say red, homomorphic image $\HH'$ of $\HH$ in $\G''$, where every vertex in $\HH'$ is the image of at most $b$ vertices in $\HH$; let $\varphi : \HH \to \G''$ be a homomorphism that maps edges of $\HH$ to red edges of $\G''$, such that every vertex in $\G''$ is the image of at most $b$ vertices in $\HH$.
		We claim that $\G'$ contains a red copy of $\HH$. To see this, define a map $\varphi' : \HH \to \G'$, where each vertex $u$ in $\HH$ is mapped to a vertex in $B'(\varphi(u))$, and moreover $\varphi'$ is injective. Such a map exists by choice of $\varphi$ and because $|B'(x)| = b$ for every vertex $x$. By choice of $\G'$, the image of $\HH$ under $\varphi'$ is a red copy of $\HH$ in $\G'$ (and so in $\G$), as required.
	\end{proof}

	We will also need the following lemma for the proof of the main result. Its proof is quite long but mundane; we thus delay it to \Cref{sec:subdivisions}.

	\begin{lem} \label{lem:reduction}
		Let $d, t \ll D, T$. The following holds for every $n$ and $\ell \ge 28t$, where $L = \ell/4t - 6$.
		For every $H \in \family_{n, d, \ell}$ there exists a bipartite $F$ in $\family_{n, D, L}$ such that $F^-[K_T]$ contains a copy of $H^t$.
	\end{lem}

	Finally, here is the proof of the main result of this paper.

	\begin{proof}[Proof of \Cref{thm:main} using \Cref{thm:main-auxiliary,lem:reduction}]
		Let $r, s, t, d \ge 1$ be integers. Let $T, D$ be as in \Cref{lem:reduction}, applied with $d$ and $t$, such that $T \ge r$. Let $\eps$ satisfy $T, s, D \ll \eps^{-1}$, and let $\Delta = 80(1/\eps) \log(1/\eps)$, so we also have $T, s, D \ll \Delta$. Finally, pick $k, B, \alpha$ so that $\eps^{-1}, \Delta \ll k \ll B \ll \alpha \ll n$ (such parameters exist because $n$ is large with respect to $r, s, t, d$).

		By \Cref{prop:existence-expanders}, there exists an $\eps$-expander $G$ on $\alpha n$ vertices with maximum degree at most $\Delta$. Let $\G = K_r(G^k[K_B])$. Then, by \Cref{thm:main-auxiliary},
		\begin{equation} \label{eqn:G}
			\G \sarrow K_r(F^-[K_T]) \qquad \text{for every bipartite $F$ in $\family_{n, D, \log(\alpha n)}$}.
		\end{equation}
		Let $H \in \family_{n, d, 10t\log n}$. By choice of $T$ and $D$ (according to \Cref{lem:reduction}), taking $L = 10t \log n / 4t - 5 = 2.5 \log n - 5$, there exists a bipartite $F$ in $\family_{n, D, L}$ such that $F^-[K_T]$ contains a copy of $H^t$, so $K_r(F^-[K_T])$ contains a copy of $K_r(H^t)$. As $n \gg \alpha$, we have $L = 2.5 \log n - 5 \ge 2 \log(\alpha n)$, implying that $F \in \family_{n, D, 2 \log(\alpha n)}$. It follows from \eqref{eqn:G} that $\G \sarrow K_r(H^t)$, completing the proof.
	\end{proof}

	\begin{rem}
		In order to prove the version of our main result for powers of tight paths, one can bypass \Cref{lem:reduction}. Indeed, take $F = P_n$ and $T = r+t-1$ in \eqref{eqn:G} and observe that $K_r(P_n^{-}[K_{r+t-1}])$ contains the $t$-power of the tight path $P_n^{(r)}$. 
	\end{rem}

	\subsection{Proof of \Cref{thm:main-homomorphic}}\label{sec:mainproof}
		Throughout the proof, we fix constants as follows. Let $h$ be such that 
		\begin{equation} \label{eqn:rst}
			r, s, t \ll h \ll \eps^{-1}, \Delta, 
		\end{equation}
		and let $d_0, \ldots, d_h, d_1', \ldots, d_{h-1}', k_0, \ldots, k_{h}$ be constants such that $k_0 = k$ and 
		\begin{equation} \label{eqn:dk}
			\eps^{-1}, \Delta \ll d_h \ll k_{h} \ll d_{h-1}'\ll d_{h-1}\ll \ldots \ll d_1' \ll d_1 \ll k_1 \ll d_0 \ll k_0.   
		\end{equation}
		Fix an $s$-colouring of $K_r(G^{k_0})$. Our aim is to show that there is a monochromatic copy of a homomorphic image $\HH'$ of $\HH = K_r(Q^-[K_{t}])$ such that every vertex in $\HH'$ is the image of at most  $(D+1) \Delta^{(2k + 1)}$  vertices in $\HH$.

		Throughout the proof we will maintain a set of vertices $U_j \subseteq V(G)$ and a collection of ordered trees $T_j(u)$, for $u \in U_j$, as in the following definition. The trees will serve as auxiliary structures that will become more complex as $j$ grows.
		\begin{defn}\label{def:tree-skeleton} 
			For $j\in [h]$, a \emph{height $j$ tree skeleton} is a pair $(U_j,\, \{T_j(u) \,|\, u \in U_j\})$ as follows.
			\begin{enumerate}[label = \rm(I\arabic*)]
				\item  \label{itm:size}
					$|U_j| \ge |G|/2^j$,
				\item  \label{itm:ordered}
					$T_j(u)$ is a $d_j$-ary ordered tree of height $j$ with $V(T_j(u))\subseteq V(G)$, for each $u\in U_j$, 
				\item  \label{itm:edges}
					the edges of $T_j(u)$ between levels $j-i$ and $j-(i-1)$ are in $G^{k_i}$, for each $u\in U_j$ and $i\in [j]$. 
			\end{enumerate}
			Given a height $j$ tree skeleton $(U_j,\, \{T_j(u) \,|\, u \in U_j\})$, define $\G_j$ to be the $r$-graph on vertices $\bigcup_{u \in U_j} L(T_j(u))$ and edges
			\begin{equation} \label{eqn:edges-hypergraph}
				\{(v_1, \ldots, v_r) \,:\, \exists \, u_1, \ldots, u_r \in U_j \text{ s.t.\ } v_i \in L(T_j(u_i)) \text{ and $\{u_1, \ldots, u_r\}$ is a clique in $G^{k_j}$}\}. 
			\end{equation}
			(The vertices $u_1, \ldots, u_r$ need not be distinct.)
		\end{defn}

		The following observation shows that the hypergraph $\G_j$ defined above is a subhypergraph of $K_r(G^{K_0})$, and so we can $s$-colour $\G_j$ according to the $s$-colouring of $K_r(G^{K_0})$. 

		\begin{obs}\label{claim:Gj} 
			Then  $\G_j \subseteq K_r(G^{k_0})$.
		\end{obs}

		\begin{proof} 
			Suppose that $(v_1, \ldots, v_r) \in E(\G_j)$ and let $u_1, \ldots, u_r$ be as in \eqref{eqn:edges-hypergraph}. Then for every $i, l \in [r]$, either $u_i u_l$ is an edge in $G^{k_j}$ or $u_i = u_l$, and so by \ref{itm:edges},
			\begin{equation*}
				\dist_G(v_i, v_l) 
				\le \dist_G(v_i, u_i) + \dist_G(u_i, u_l) + \dist_G(u_l, v_l)
				\le 2(k_1 + \ldots + k_j) + k_j 
				\le k_0,
			\end{equation*}
			(using $k_1, \ldots, k_j \ll k_0$). It follows that $(v_1, \ldots, v_r) \in E(K_r(G^{k_0}))$, as required.	
		\end{proof}

		The following proposition is the main drive of the proof.

		\begin{prop} \label{prop:one-step}
			Let $j \in \{0, \ldots, h-1\}$. Suppose that $(U_j, \, \{T(u) \,|\, u\in U_j\})$ is a height $j$ tree skeleton with the following property.

			\begin{enumerate} [label = \rm(MT)]
				\item \label{itm:MT}
					There are no disjoint sets $X, Y, Z\subseteq V(G)$ of size $t$ such that for some $u \in U_j$ we have: $X, Z \subseteq L(T_j(u))$; the ordered subtrees in $T_j(u)$ corresponding to $X$ and $Z$, respectively, are vertex-disjoint isomorphic ordered trees; and both $X \cup Y$ and $Y \cup Z$ are monochromatic cliques of the same colour in $\G_j$.
			\end{enumerate}

			Then either there is a height $j+1$ tree skeleton that also satisfies \ref{itm:MT}, or there is a monochromatic homomorphic copy $\HH'$ of $\HH$ such that every vertex in $\HH'$ is the image of at most $(D+1)\Delta^{2k_0+1}$ vertices in $\HH$. 
		\end{prop}

		It is easy to complete the proof using the above proposition. To see this, let $T_0(u)$ to be the tree on the single vertex $u$, for every $u \in V(G)$, and let $U_0=V(G)$. Then $(U_0,\, \{T_0(u) \,|\, u\in U_0\})$ is a height $0$ tree skeleton that satisfies \ref{itm:MT} trivially (as $L(T_j(u))$ has size $1$).\footnote{The parameter $d_0$ in the definition of a tree skeleton plays a figurative role here, and can be replaced by any other constant.}

		By iterating \Cref{prop:one-step} up to $h$ times, either there is a monochromatic homomorphic image $\HH'$ of $\HH$ such that every vertex in $\HH'$ is the image of at most $(D+1) \Delta^{2k_0+1} $ vertices in $\HH$, completing the proof of \Cref{thm:main-auxiliary}; or there is a height $h$ tree skeleton that satisfies \ref{itm:MT}, a contradiction to \Cref{lem:trees-tight-paths}.

		\begin{proof} [Proof of \Cref{prop:one-step}]
			We first prove that there is a collection of ordered $d_j'$-ary subtrees $\{ T'(u_j) \subseteq T(u_j) \,|\, u\in U_j\}$ that satisfies the following two properties.
			\begin{enumerate} [label = \rm(J\arabic*)]
				\item  \label{itm:order}
					for every edge $uv$ of $G^{3k_{j+1}}$, the trees $T_j'(u)$ and $T_j'(v)$ are vertex-disjoint,
				\item  \label{itm:mono}
					for every clique $\{u_1, \ldots, u_{\rho}\}$ in $G^{3k_{j+1}}$ with $\rho \le r$, any two $r$-sets $X, Y \subseteq \bigcup_{i \in [\rho]} L(T_j'(u_i))$, such that $X \cap L(T_j(u_i))$ and $Y \cap L(T_j(u_i))$ correspond to isomorphic ordered trees for $i \in [\rho]$, have the same colour.
			\end{enumerate}
			Note that if $j = 0$, we can take $T_j'(u) = T_j(u)$ for every $u \in U_0$; properties \ref{itm:order} and \ref{itm:mono} trivially hold in this case. 

			To see that such trees $T_j'(u)$ exist for $j \ge 1$, recall that $k_{j+1} \ll d_j' \ll d_j$ (see \eqref{eqn:dk}).  Apply \Cref{lem:trimming-order}, with parameters $d_j'$, $j$, $r$, $(\Delta+1)^{3k_{j+1}}$, $d_j$, and the graph $G^{3k_{j+1}}$. As $d_j', j, k_{j+1}, \Delta \ll d_j$, and the maximum degree of $G^{3k_{j+1}}$ is bounded by $(\Delta + 1)^{3k_{j+1}}$, the lemma is applicable. It yields a collection $\{T_j''(u) \subseteq T_j(u) \,|\, u \in U_j\}$, of $d_j'$-ary subtrees of height $j$, that satisfies \ref{itm:order}.

			Next, we can apply \Cref{lem:trimming-ramsey} with parameters $d_{j+1}, r, j, (\Delta+1)^{3r k_{j+1}}, d_j'$, the hypergraph whose edges are cliques in $G^{3k_{j+1}}$ of size at most $r$, and the collection of trees $\{T_j''(u) \,|\, u\in V(G)\}$. As before, because $d_{j+1}, j, k_{j+1}, \Delta \ll d_j'$, and the maximum degree of the aforementioned hypergraph is at most $(\Delta+1)^{3rk_{j+1}}$, the lemma is applicable. It yields a collection $\{T_j'(u) \subseteq T_j''(u) \,|\, u\in V(G)\}$, of $d_{j+1}$-ary subtrees of height $j$, that satisfies \ref{itm:mono}. 

			Let $\G_j'$ be the hypergraph on vertices $\bigcup_{u \in U_j} L(T_j'(u))$ and edges
			\begin{equation*}
				\{(v_1, \ldots, v_r) \,:\, \exists \, u_1, \ldots, u_r \in U_j \text{ s.t.\ } v_i \in L(T_j'(u_i)) \text{ and $\{u_1, \ldots, u_r\}$ is a clique in $G^{3k_{j+1}}$}\}. 
			\end{equation*}
			Note that $\G_j' \subseteq \G_j$ since $k_{j+1}\ll k_j$, and $\G_j\subseteq  K_r(G^{k_0})$ by \Cref{claim:Gj}. Consider the $s$-colouring of $\G_j'$ inherited from $K_r(G^{k_0})$.

			We define an auxiliary edge-colouring of $G[U_j]^{k_{j+1}}$ as follows: for an edge $uv$, if there exist disjoint sets $X, Y, Z$ of size $t$ such that 
			\begin{itemize}
				\item
					$X \subseteq L(T_j'(u))$, $Y \subseteq \bigcup_{w \in U_j} L(T_j'(w))$ and $Z \subseteq L(T_j'(v))$,
				\item
					the ordered subtrees of $T_j'(u))$ and $T_j'(v))$ that correspond to $X$ and $Z$, respectively, are isomorphic copies of some ordered tree $S$,
				\item
					$\G_j'[X \cup Y]$ and $\G_j'[Y \cup Z]$ are both monochromatic cliques of some colour $c$,
			\end{itemize}
			then colour $uv$ by $(c, S)$. If no such $X, Y, Z$ exist, colour $uv$ grey. (It may be that an edge $uv$ receives more than one non-grey colour.)

			Let $s'$ be the number of pairs $(c, S)$ as above; then $s'$ is bounded by a function of $s, t, h$. Apply \Cref{cor:ramsey-expanders} with parameters $s', d_{j+1}+1, D, 2^{-j}, n$ and $\eps, k_{j+1}, \alpha$, and graph $G$. As $j \le h \ll \eps^{-1}$ and $r, s, t, D, h, d_{j+1} \ll k_{j+1} \ll \alpha$, the corollary is applicable. It implies that
			\begin{equation} \label{eqn:ramsey-auxiliary}
				G[U_j]^{k_{j+1}} \to 
				\left(\,\, 
				\overbrace{K_{d_{j+1}+1} \cup \ldots \cup K_{d_{j+1}+1}  }^{\text{disjoint cliques covering $|U_j|/2$ vertices}}, \quad \quad \overbrace{Q, \ldots, Q}^{s'} 
				\,\,\right).
			\end{equation}

			We consider two cases: there is a non-grey copy of $Q$; or there is a collection of pairwise-disjoint grey $(d_{j+1}+1)$-cliques that covers at least $|U_j|/2$ vertices.

		\subsubsection*{Case 1: non-grey monochromatic $Q$}
			In this case there exists a $(c, S)$-coloured copy of $Q$, denoted by $Q'$, for some $c$ and $S$ as above. We shall need the following claim, that guarantees the existence of an `extensible' copy of $S$ in each tree $T_j'(u)$.

			\begin{claim} \label{claim:find-good-Su}
				For every $u \in U_j$, there is a copy $S_u$ of $S$ in $T_j'(u)$ such that: 
				\begin{itemize}
					\item
						$L(S_u) \subseteq L(T_j'(u))$, 
					\item
						for every ordered tree $S'$ that contains $S$, has height $j' \le j$, has at most $2t$ leaves, and $L(S) \subseteq L(S')$, there is a copy of $S'$ in $T_j'(u)$ that contains $S_u$.
				\end{itemize}
			\end{claim}

			\begin{proof}
				For an ordered tree $T$ with root $x$, denote by $T_u$ the subtree of $T$ rooted at $u$.
				A \emph{respectful labelling} of $T$ is a labelling $\varphi: V(T) \setminus \{x\} \to \N$ such that for any two vertices $y$ and $z$ that are children of the same vertex, if the leaves in the subtree $T_y$ precede the leaves in $T_z$ then $\varphi(y) < \varphi(z)$ (note that by definition of an ordered tree, either the leaves of $T_y$ precede the leaves of $T_z$, or vice versa).

				Fix $u \in U_j$ and denote $T = T_j'(u)$. Let $\varphi$ be a respectful labelling of $T$ with labels in $[d_j']$ (note that all labels are uniquely determined). Let $\psi$ be a respectful labelling of $S$ with labels in $\{t+1, 2(t+1), \ldots, t(t+1)\}$; such a labelling exists, because $S$ has $t$ leaves and so every vertex in $S$ has at most $t$ children. 

				Let $S'$ as in the statement of the claim and fix an embedding $f : S \to S'$. We claim that there is a respectful labelling $\psi'$ of $S'$ with labels in $[t^2+2t]$ that extends $\psi$; namely $\psi(u) = \psi'(f(u))$ for every $u \in S$ (except for the root of $S$, which does not receive a label in $\psi$).
				Indeed, this follows by choice of $\psi$ because every vertex in $S$ has at most $t$ children that are not in $S$.

				We take $S_u$ to be the copy of $S$ in $T$ for which the labelling $\varphi$ restricted to $S_u$ is $\psi$. The claim follows from the previous paragraph. (Here we use that $d_j' \gg t$ to guarantee the existence of $S_u$.)
			\end{proof}

			For each $u \in U_j$, fix a copy $S_u$ of $S$ in $T_j'(u)$ as in \Cref{claim:find-good-Su}.
			By definition of the auxiliary colouring of $G[U_j]^{k_{j+1}}$, for every edge $uv$ in $Q'$, there are disjoint sets $X_{uv}, Y_{uv}, Z_{uv}$ of size $t$ such that: $X_{uv} \subseteq L(T_j'(u))$ and $Z_{uv} \subseteq L(T_j'(v))$,  and the ordered subtrees of $T_j'(u)$ and $T_j'(v)$ corresponding to $X_{uv}$ and $Z_{uv}$ are isomorphic to $S$.  Moreover,  $X_{uv} \cup Y_{uv}$ and $Y_{uv} \cup Z_{uv}$ are $c$-coloured cliques in $\G_j'$. We claim that we may assume $X_{uv} = L(S_u)$ and $Z_{uv} = L(S_v)$. 

			Indeed, replace $X_{uv}$ by $L(S_u)$, and replace the vertices of $Y_{uv} \cap L(T_j'(u))$ by a set $Y'$ of the same size in $L(T_j'(u))$, such that the subtree of $T_j'(u)$ corresponding to $L(S_u) \cup Y'$ is isomorphic to the subtree of $T_j'(u)$ corresponding to $(X_{uv} \cup Y_{uv}) \cap L(T_j'(u))$. Such a set $Y'$ exists due to the choice of $S_u$ as in \Cref{claim:find-good-Su}. These modified sets $X_{uv}, Y_{uv}, Z_{uv}$ satisfy the same properties as the original sets, due to \ref{itm:mono}. Similarly, we may assume that $Z_{uv} = L(S_v)$. 

			Consider the $r$-uniform hypergraph $\HH'$ on vertices $\left(\bigcup_{u \in V(Q')} L(S_u)\right) \cup \left(\bigcup_{uv \in E(Q')} Y_{uv}\right)$ whose edges are $r$-subsets of $L(S_u) \cup Y_{uv}$, for $uv \in E(Q')$. Note that this is a $c$-coloured homomorphic copy of $\HH = K_r(Q^-[K_{t}])$.
			Notice also that every vertex in $L(S_u)$ is at distance at most $k_0$ from $u$, by \ref{itm:edges}. Similarly, every vertex in $Y_{uv}$ is at distance at most $2k_0$ from $u$. Indeed, since $L(S_u) \cup Y_{uv}$ is a clique in $\G_j'$ and $\G_j' \subseteq K_r(G^{k_0})$, every vertex in $Y_{uv}$ is at distance at most $k_0$ from every vertex in $L(S_u)$, which in turn is at distance at most $k_0$ from $u$. It follows that every vertex in $G$ is in at most $1 + \Delta + \ldots + \Delta^{k_0} \le \Delta^{k_0+1}$ sets $L(S_u)$, and in at most $D \cdot (1 + \Delta + \ldots + \Delta^{2k_0}) \le D \cdot \Delta^{2k_0+1}$ sets $Y_{uv}$. So $\HH'$ is a $c$-coloured homomorphic copy of $\HH$ where every vertex in $\HH'$ is the image of at most $(D + 1) \cdot \Delta^{2k_0+1}$ vertices in $\HH$, as required.

		\subsubsection*{Case 2: many disjoint grey cliques}

			In this case there is a collection of pairwise disjoint grey $(d_{j+1}+1)$-cliques in $G^{k_{j+1}}[U_j]$ that covers at least $|U_j|/2$ vertices in $U_j$; denote by $U_{j+1}$ the set of vertices covered by these cliques.

			Given $u \in U_{j+1}$, let $K$ be a grey $(d_{j+1}+1)$-clique in $G^{k_{j+1}}$ that contains $u$, and let $u_1, \ldots, u_{d_{j+1}}$ be the other vertices in $K$. Form a $d_{j+1}$-ary tree $T_{j+1}(u)$ of height $j+1$ by joining the roots of $T_j'(u_1), \ldots, T_j'(u_{d_{j+1}})$ to $u$ and thinking of $u$ as the root. Note that $T_{j+1}(u)$ is indeed a tree since $T_j'(u_p)$ and $T_j'(u_q)$ are vertex disjoint for $p\neq q$ due to \ref{itm:order} and since $u_pu_q\in E(G^{k_{j+1}}[U_j])$. Inside each $T_j'(u_q)$, keep the order of the leaves in $T_{j+1}(u)$ as they were in $T_j'(u)$, and order the leaves of distinct trees in increasing order of the indices of the roots, that is, we first put the leaves of $T_j'(u_1)$ then the leaves of $T_j'(u_2)$, etc.

			\begin{claim} 
				$(U_{j+1},\, \{T_{j+1}(u) \,|\, u\in U_{j+1}\})$ is a $d_{j+1}$-ary tree skeleton of height $j+1$ that satisfies \ref{itm:MT}.
			\end{claim}

			\begin{proof} 
				It is easy to check that \ref{itm:size} to \ref{itm:edges} hold. It remains to check that \ref{itm:MT} holds too.  

				To do so, we will use the fact that $\G_{j+1} \subseteq \G_j' \subseteq \G_j$. The second inclusion is true since $k_{j+1}\ll k_j$. To see that $\G_{j+1} \subseteq \G_j'$, consider an edge $(v_1, \ldots, v_r)$ in $\G_{j+1}$, and let $u_1, \ldots, u_r \in U_{j+1}$ be such that $v_i \in L(T_{j+1}(u_i))$ and $\{u_1, \ldots, u_r\}$ is a clique in $G^{k_{j+1}}$. (Recall that the $u_i$'s need not be distinct.) By definition of $T_{j+1}(\cdot)$, there exist $w_1, \ldots, w_r \in U_j$ such that $u_i w_i \in E(G^{k_{j+1}})$ and $v_i \in L(T_j'(w_i))$. Since $\{u_1, \ldots, u_r\}$ is a clique in $G^{k_{j+1}}$ it follows that $\{w_1, \ldots, w_r\}$ is a clique in $G^{3k_{j+1}}$, so $(v_1, \ldots, v_r) \in E(\G_j')$.

				Suppose now that \ref{itm:MT} does not hold for $(U_{j+1},\, \{T_{j+1}(u) \,|\, u\in U_{j+1}\})$. Then there exist disjoint sets $X, Y, Z$ of size $t$ and $u \in U_j$ such that $X, Z \subseteq L(T_{j+1}(u))$; the subtrees $S_X$ and $S_Z$ of $T_{j+1}(u)$ corresponding to $X$ and $Z$ are isomorphic and vertex-disjoint; and both $X \cup Y$ and $Y \cup Z$ are monochromatic cliques in $\G_{j+1}$. 

				Note that, by disjointness of $S_X$ and $S_Z$, the roots of $S_X$ and $S_Z$ are not $u$. We can thus define $w_X$ and $w_Z$ to be the children of the root of $T_{j+1}(u)$ that are common ancestors of $X$ and $Z$, respectively.

				If $w_X = w_Z$, then this contradicts \ref{itm:MT} for $j$, using $\G_{j+1} \subseteq \G_j$. If $w_X \neq w_Z$, then $w_X w_Z$ is a grey edge of $G^{k_{j+1}}$, by construction of $T_{j+1}(u)$. We thus reached a contradiction to the definition of a grey edge, as $\G_{j+1} \subseteq \G_j'$.
				It follows that \ref{itm:MT} holds, as required.
			\end{proof}

			This completes the proof of \Cref{prop:one-step}. With it, the proof of \Cref{thm:main-auxiliary} is also complete.
		\end{proof}

\section{Powers of hypergraph trees} \label{sec:reductiontrees}

	Recall \Cref{def:hypergraph-tree}, which defines an \emph{$r$-uniform tree} to be an $r$-graph with edges $\{e_1, \ldots, e_m\}$ such that for every $i \in \{2, \ldots, m\}$ the following holds: $|e_i \cap \bigcup_{1 \le j < i} e_j| \le r-1$ and $e_i \cap \bigcup_{1 \le j < i} e_j \subseteq e_{i_0}$ for some $i_0 \in [i-1]$. Our aim in this lemma is to deduce the version of our main result for powers of $r$-uniform trees, namely \Cref{thm:main-hypergraph-tree}, from \Cref{thm:main-tree}. The main ingredient in this deduction is the following lemma, which we prove below. Here $c(\T)$ is the number of connected components in $\T$.

	\begin{lem} \label{lem:switch-to-hypergraph-trees}
		Let $\T$ be an $r$-uniform tree on $n$ vertices with maximum degree at most $d$.
		There exists a tree $S$ on $n+c(\T)$ vertices with maximum degree at most $d \cdot r$ such that $\T \subseteq K_r(S^{d+1})$.
	\end{lem}

	The following corollary follows quite easily (recall \Cref{def:power} about powers of hypergraphs).

	\begin{cor} \label{cor:power-tight-trees}
		Let $\T$ be an $r$-uniform tree on $n$ vertices with maximum degree at most $d$. Then 
		there is a tree $S$ on $n+c(\T)$ vertices with maximum degree at most $d \cdot r$ such that ${\T}^t \subseteq K_r(S^{t(d+1)})$ for any $t \ge 1$.
	\end{cor}

	\begin{proof}
		By Lemma~\ref{lem:switch-to-hypergraph-trees} there is a tree $S$ on $n+c(\T)$ vertices with maximum degree at most $d \cdot r$ such that $\T \subseteq K_r(S^{d+1})$. Let $e \in E(\T^t)$. By definition of $\T^t$, there exists a tight path $P$ on at most $r+t-1$ vertices that contains $e$; denote the vertices of $P$ (in order) by $(u_1, \ldots, u_s)$. Observe that for $i, j \in [s]$ with $|i-j| \le r-1$, the vertices $u_i$ and $u_j$ are contained in an edge of $\T$, so by choice of $S$ they are at distance at most $d+1$ in $S$. It follows that $\dist_S(u_i, u_j) \le \ceil{(s-1)/(r-1)} \cdot (d+1) \le t (d+1)$ for every $i, j \in [s]$. This implies that $\{u_1, \ldots, u_s\}$ is a clique in $S^{t(d+1)}$. As $e \subseteq \{u_1, \ldots, u_s\}$, we have $e \in K_r(S^{t(d+1)})$, establishing that $\T^t \subseteq K_r(S^{t(d+1)})$, as required. 
	\end{proof}

	It is easy to prove the version of our main result for hypergraph trees.

	\begin{proof} [Proof of \Cref{thm:main-hypergraph-tree}]
		Let $\T$ be an $r$-uniform tree on $n$ vertices with maximum degree $d$. By \Cref{cor:power-tight-trees}, there is a tree $S$ on $n + c(\T) \le 2n$ vertices with maximum degree at most $dr$, such that $\T^t \subseteq K_r(S^{t(d+1)})$. By \Cref{thm:main-tree}, applied with $r, s, dr, t(d+1)$, we find that $\rhat_s(K_r(S^{t(d+1)})) = O(n)$. It follows that $\rhat_s(\T^t) = O(n)$, as required.
	\end{proof}

	It remains to prove \Cref{lem:switch-to-hypergraph-trees}.

	\begin{proof}[Proof of \Cref{lem:switch-to-hypergraph-trees}]

		Observe that the connected components of any $r$-uniform tree are themselves $r$-uniform trees. It thus suffices to prove the lemma under the assumption that $\T$ is connected. Indeed, if $\T$ consists of components $\T_1, \ldots, \T_k$ and the statement for connected hypergraph trees holds, then there are trees $S_1, \ldots, S_k$ such that $S_i$ has at most $|\T_i|+1$ vertices, has maximum degree at most $dr$, and satisfies $\T_i \subseteq K_r(S^{d+1})$, for $i \in [k]$. 
		Form a tree $S$ as follows. Suppose that $S_1, \ldots, S_k$ are pairwise vertex-disjoint. For $i \in [k-1]$, let $e_i$ be an edge that joins a leaf of $S_i$ with a leaf of $S_{i+1}$. Take $S$ to be the union of the trees $S_1, \ldots, S_k$ and the edges $e_1, \ldots, e_{k-1}$; then $S$ is a tree on $n + c(\T)$ vertices, with maximum degree at most $dr$ (observe that the degree of vertices incident with at least one edge $e_i$ is at most $dr$, and we may assume $dr \ge 3$, as other cases result in trivial statements). From now on, we assume that $\T$ is connected.

		Let $e_1, \ldots, e_m$ be the edges of $\T$, and suppose that for every $i \in \{2, \ldots, m\}$ we have $1 \le |e_i \cap (\bigcup_{1 \le j < i} e_j)| \le r-1$ (using that $\T$ is connected), and $e_i \cap (\bigcup_{1 \le j < i}e_j) \subseteq e_{i_0}$ for some $i_0$ with $1 \le i_0 < i$. 
		Define $f: [m] \to [m]$ by setting $f(1) = 1$ and $f(i) = i_0$, where $i > 1$ and $i_0$ is as above. We will denote $f(f(i))$ simply by $f^2(i)$.

		For $i \in [m]$, let $\T_i$ be the subtree of $\T$ spanned by the edges $e_1, \ldots, e_i$. We construct a tree $S_i$ on the vertex set $\{x\} \cup V(\T_i)$ for every $i \in [m]$, where $x$ is a new vertex, and such that $S_1 \subseteq \ldots \subseteq S_m = S$, as follows.

		\begin{itemize}
			\item 
				$S_1$ is a star whose root is $x$ and whose leaves are the vertices in $e_1$,
			\item  
				for each $i$ with $2 \le i \le m$, pick $p_i$ as follows: if $f(i) = 1$, define $p_i$ to be any vertex from $ e_i \cap e_1$; and if $f(i) > 1$, define $p_i$ to be any vertex from $ (e_i \cap e_{f(i)}) \setminus e_{f^2(i)}$. Form $S_i$ by adding the vertices in $e_i \setminus e_{f(i)}$ to $S_{i-1}$ and joining them to $p_i$.
		\end{itemize}
		(Note that if $f(i) > 1$, then $(e_i \cap e_{f(i)}) \setminus e_{f^2(i)}$ is non-empty by minimality of $f(i)$, as $f^2(i) < f(i)$.)

		It is easy to see that each $S_i$ is a tree, as we keep adding leaves at each iteration, starting  with a star. 
		We first argue that $S$ has bounded degree, using the following claim.
		\begin{claim} \label{claim:deg-Si}
			For every vertex $u$ in $\T_i$ we have $\deg_{S_i}(u) \le r \cdot \deg_{\T_i}(u)$.
		\end{claim}
		\begin{proof} 
			We prove the statement by induction on $i$.
			The statement holds for $i=1$, as the vertices of $\T_1$ have degree $1$ in both $S_1$ and $\T_1$. 
			For $i \ge 2$, suppose that the statement holds for $i-1$. Note that the only vertices whose degree changes when moving from $S_{i-1}$ to $S_i$ are $p_i$, whose degree increases by less than $r$, and the elements of $e_i \setminus e_{f(i)}$, whose degree in $S_i$ is $1$. As $\deg_{\T_i}(p_i) = \deg_{\T_{i-1}}(p_i) - 1$ (because $p_i \in e_i$), the statement for $i$ follows. 
		\end{proof}
		Note that $\deg_{S_i}(x) = r$ for every $i \in [m]$. It thus follow from \Cref{claim:deg-Si} that $\Delta(S) \le rd$. 

		It remains to show that for every edge $e_i$ in $\T$, the distance between any two elements of $e_i$ is at most $d + 1$. We use the following claim.

		\begin{claim} \label{claim:distance-Si}
			Let $i \geq 2$, $u \in e_i \setminus e_{f(i)}$ and $w \in e_i$. Then $\dist_{S}(u, w) \le \deg_{\T_i}(w) + 1$.
		\end{claim}

		\begin{proof}
			We prove the claim by induction on $\deg_{\T_i}(w)$. If $\deg_{\T_i}(w)=1$, this means that $w \in e_i\setminus e_{f(i)}$, which by construction of $S$ implies $\dist_{S}(u, w) \leq  2 =\deg_{\T_i}(w) + 1$, as desired.

			Now suppose that $\deg_{\T_i}(w) > 1$. In particular, we have $w \in e_{f(i)}$. 
			We consider two cases: $f(i) = 1$ and $f(i) > 1$. 
			In the former case, we have $w, p_i \in e_1$, and so $\dist_{S}(w, p_i) = 2$, implying that $\dist_{S}(u, w) \le \dist_{S}(u, p_i) + \dist_{S}(p_i, w) = 3 \le \deg_{\T_i}(w) + 1$.
			Now suppose that $f(i) > 1$. Then $\deg_{\T_{f(i)}}(w) < \deg_{\T_{i}}(w)$ and $p_i \in e_{f(i)} \setminus e_{f^2(i)}$. By induction, applied to $f(i)$, $p_i$, and $w$, we find that $\dist_{S}(p_i, w) \le \deg_{\T_{f(i)}}(w) + 1 \leq  \deg_{\T_{i}}(w)$. As $p_i$ and $u$ are adjacent in $S$, it follows that $\dist_{S}(u, w) \le   \dist_{S}(p_i, w) + \dist_{S}(u, p_i)\le \deg_{\T_{i}}(w) + 1$, as required. 
		\end{proof}

		It is now easy to deduce that for every $i \in [m]$ we have $\dist_S(u, w) \le d+1$ for every $u, w \in e_i$. Indeed, for such $i, u, w$, without loss of generality, $i$ is minimum such that $u, w \in e_i$. Then either $u, w \in e_1$ and so $\dist_S(u, v) = 2$, or at least one of $u$ and $w$ is in $e_i \setminus e_{f(i)}$. Either way, by \Cref{claim:distance-Si}, we have $\dist_S(u, w) \le d+1$, as required.
	\end{proof}

\section{Conclusion} \label{sec:conc}

    In this paper we studied families of bounded degree hypergraphs whose size-Ramsey number is linear in their order. Not much is known about bounded degree hypergraphs, or even graphs, whose size-Ramsey numbers are not linear. As mentioned in the introduction, R\"odl and Szemer\'edi \cite{rodl2000size} constructed a sequence $(H_n)$, where $H_n$ is an $n$-vertex graph with maximum degree $3$ such that $\rhat(H_n) = \Omega(n (\log n)^{1/60})$, thus refuting a conjecture of Beck \cite{beck1983size}, which said that all bounded degree graphs have linear size-Ramsey numbers.  (This construction can be generalized to hypergraphs, as shown in~\cite{dudek2017size}.) R\"odl and Szemer\'edi  conjectured that this bound can be improved to $n^{1 + \eps}$ for some constant $\varepsilon >0 $, and this is widely open. We restate their conjecture here.
    
    \begin{conj}
        There exist $\Delta$, $\eps > 0$ and a sequence $(F_n)$, where $F_n$ is an $n$-vertex graph with maximum degree at most $\Delta$, such that $\rhat(F_n) = \Omega(n^{1+\eps})$.
    \end{conj}
    
    Regarding upper bounds, Kohayakawa, R\"{o}dl, Schacht, and Szemer\'edi~\cite{kohayakawa2011sparse} proved that for every $\Delta\geq 3$ and $\eps <1/\Delta$, every $n$-vertex graph $F$ with maximum degree at most $\Delta$ satisfies $\rhat(F) = O(n^{2-\eps})$, thus answering a question of R\"odl and Szemer\'edi \cite{rodl2000size}. 
    
    In fact, other than the R\"odl--Szemer\'edi construction, we do not know of any other examples of bounded degree graphs with superlinear size-Ramsey number. One candidate that seems worth considering is the grid graph. The $n\times n$ \emph{grid} graph $G_{n,n}$ is defined on the vertex set $[n]\times [n]$ with edges $uv$ present, whenever $u$ and $v$ differ in exactly one coordinate and the difference is exactly $1$ (equivalently, $G_{n, n} \cong P_n \,\square\, P_n$, where $H \,\square\, F$ is the Cartesian product of $H$ and $F$). Recently, Clemens, Miralaei, Reding, Schacht and Taraz~\cite{clemens2021size} showed that the size-Ramsey number of the grid graph on $n\times n$ vertices is bounded from above by $n^{3+o(1)}$. As is often the case, the host graph in their proof is a random graph with $\Theta(n^2)$ vertices and appropriate density, and the $n^{3+o(1)}$ bound is tight, up to the $o(1)$ term in the exponent, for random graphs. No non-trivial lower bounds are known. It would thus be very interesting to have an answer to the following question.
    
    \begin{qn}
        Is the size-Ramsey number of the grid graph on $n \times n$ vertices $O(n^2)$?
    \end{qn}
    
   For hypergraphs,  even less is known. Dudek, La Fleur, Mubayi and R\"odl \cite{dudek2017size} asked for the maximum size-Ramsey number of an $r$-uniform $\ell$-tree (imposing no restrictions on the maximum degree). Recall that for $r\geq 3$ and $1\leq \ell \leq r-1$, an \emph{$r$-uniform $\ell$-tree} is an $r$-graph with edges $\{e_1, \ldots, e_m\}$ such that for every $i \in \{2, \ldots, m\}$ we have $\big|e_i \cap (\bigcup_{1 \le j < i} e_j)\big| \le \ell$ and $e_i \cap (\bigcup_{1 \le i < j} e_j) \subseteq e_{i_0}$ for some $i_0 \in [i-1]$.  The authors in~\cite{dudek2017size} showed that if $\T$  is an $r$-uniform $\ell$-tree then   $\rhat(\T)=O(n^{\ell+1})$, which is tight for $\ell=1$.  They asked whether the bound is tight for all $\ell$. We suspect the answer to be positive, and reiterate their question here.
    \begin{qn} 
        For any  $r\geq 3$ and $1\leq \ell \leq r-1$, is it true that for every $n$ there exists $r$-uniform 
        $\ell$-tree $\T$ of order at most $n$ such that $\rhat(\T)=\Omega(n^{\ell+1})$?
    \end{qn}

    Another interesting problem is the tightness of known bounds on the size-Ramsey number of $q$-subdivisions of bounded degree graphs, where $q$ is fixed. Recall that Dragani\'c, Krivelevich and Nenadov \cite{draganic2021size} recently showed that for fixed $\Delta$ and $q$, the size-Ramsey number of the $q$-subdivision of an $n$-vertex graph with maximum degree $\Delta$ is bounded by $O(n^{1+1/q})$. Similarly to the grid, this is close to tight if the host graph is a random graph. However, it is unclear if this bound is anywhere near tight in general.  We pose it as our final question. 
    \begin{qn}
        For fixed $\Delta$ and $q$, is there a sequence $(H_n)$, where $H_n$ is an $n$-vertex graph with maximum degree at most $\Delta$, such that the size-Ramsey number of the $q$-subdivision of $H_n$ is at least $\Omega(n^{1+1/q})$? 
    \end{qn}
    
    \subsection*{Acknowledgements}
    
        We would like to thank Nemanja Dragani\'c for pointing out that their original proof of  Lemma 2.4 in \cite{draganic2021rolling} was incorrect, and that their amended proof required a slightly different version of  \Cref{def:good-embedding}. This affected our proof of \Cref{thm:embed-H}.
    
	\bibliography{size-ramsey}
	\bibliographystyle{amsplain}

	\appendix 

\section{Proof of Proposition~\ref{prop:existence-expanders}} \label{sec:appendix} 

	\begin{proof}
		Without loss of generality, $\eps < 1/2e$.
		Let $G$ be a copy of $G(N, p)$, where $N = 2\alpha n$, $p = \beta / N$ and $\beta = 20 \cdot (1/\eps) \log(1/\eps)$.

		We show that with high probability $G$ is an $\eps/2$-expander.  
		\begin{align*}
			& \Pr[\text{there exist disjoint $A, B \subseteq V(G)$ of size $\eps N/2$ with no edge from $A$ to $B$}] \\
			& \qquad \qquad = \sum_{\substack{\text{$A, B \subseteq V(G)$ disjoint,}\\ |A| = |B| = \eps N / 2}} (1 - p)^{|A| \cdot |B|} \\
			& \qquad \qquad \le \binom{N}{\eps N/2}^2 \cdot \exp(-p \cdot \eps^2 N^2 / 4) \\
			& \qquad \qquad \le (2e/\eps)^{2\eps N} \cdot \exp(-(\beta \eps^2 / 4) N) \\
			& \qquad \qquad \le \exp\left( (2\eps \log(2e/\eps) - \beta \eps^2 / 4) \cdot N\right) \\
			& \qquad \qquad \le \exp \left( (4 \log(1/\eps) - \beta \eps / 4) \cdot \eps N \right) \\
			& \qquad \qquad \le \exp \left( -\eps \log(1/\eps) \alpha n \right) = o(1).
		\end{align*}

		We also observe that as the expected number of edges in $G$ is $p \binom{N}{2} \le pN^2/2$, with high probability $G$ has at most $pN^2$ edges, e.g.\ using Chernoff's bounds. 

		We may thus take $G$ to be a graph on $2\alpha n$ vertices that satisfies the above two properties; namely, it is an $\eps/2$-expander and it has at most $p N^2$ edges. The latter property implies that there are at most $N/2$ vertices with degree larger than $4pN$. It follows that there is an induced subgraph $G'$ of $G$ that has exactly $N/2 = \alpha n$ vertices, and which has maximum degree at most $4pN = 4 \beta = 80 \cdot (1/\eps) \log(1/\eps)$. Note that $G'$ is an $\eps$-expander, so it satisfies the requirements.
	\end{proof}

\section{Proof of Theorem~\ref{thm:embed-H}} \label{sec:expanding-thm}

	Our proof of \Cref{thm:embed-H} uses machinery that was introduced by Friedman and Pippenger \cite{friedman1987expanding} to prove that $(2m, d+1)$-expanding graphs contain every tree in $\T_{m,d}$. Their method was modified by Glebov, Johannsen and Krivelevich \cite{glebov2013dissertation,glebov2021hitting} to provide a flexible approach for embedding bounded degree trees. Here we use notation and results due to Dragani\'c, Krivelevich and Nenadov \cite{draganic2021rolling}.

    In a graph $G$, given a set of vertices $X$, we write $\Gamma_G(X)$ for the set of vertices in $G$ that are neighbours of at least one vertex in $X$. 
    
    	\begin{defn} \label{def:good-embedding}
		Let $m, d$ be integers. Given graphs $G$ and $H$, an embedding $\varphi : H \to G$ is an \emph{$(m, d)$-good embedding} if for every subset $X \subseteq V(G)$ of size at most $m$ the following holds.
		\begin{equation} \label{eqn:good-embedding}
			\left| \Gamma_G(X) \setminus \varphi(V(H)) \right| \ge
			\sum_{x \in X} \left( d - \deg_H(\varphi^{-1}(x)) \right) + |\varphi(V(H)) \cap X|,
		\end{equation}
		where if $x \notin \varphi(V(H))$ then $\deg_H(\varphi^{-1}(x))$ is defined to be $0$.
	\end{defn}

    The following theorem is very similar to a result implicit in \cite{friedman1987expanding}; its proof can be found in \cite{draganic2021rolling} (see Theorem 2.3).
    
	\begin{thm} \label{thm:extend-good-embedding}
		Let $m, d$ be integers. Let $H$ be a graph on fewer than $m$ vertices and with maximum degree at most $d$, let $G$ be a $(2m-2, d+2)$-expanding graph, and let $\varphi : H \to G$ be a $(2m-2, d)$-good embedding. 
		Then for every graph $H'$ on at most $m$ vertices and with maximum degree at most $d$, that can be obtained from $H$ by successively adding vertices of degree $1$, there exists a $(2m-2, d)$-good embedding $\varphi' : H' \to G$ that extends $\varphi$.
	\end{thm}

	Next, we give three handy observations.

	\begin{obs} \label{obs:single-vx-embedding}
		Let $m, d$ be integers. Let $G$ be an $(m, d+2)$-expanding graph, and let $H$ be a graph on a single vertex. Then every embedding $\varphi : H \to G$ is $(m, d)$-good.
	\end{obs}

	\Cref{obs:single-vx-embedding} follows immediately from the definition of a good embedding.	

	\begin{lem} \label{lem:trim-good-embedding}
		Let $m, d$ be integers. Let $\varphi : H \to G$ be an $(m, d)$-good embedding. Then for every graph $H'$ obtained by successively removing vertices of degree $1$ from $H$, the restriction $\varphi'$ of $\varphi$ to $H'$ is an $(m, d)$-good embedding.
	\end{lem}

	For a proof of \Cref{lem:trim-good-embedding} see Lemma 2.4 in \cite{draganic2021rolling}. 
	
	\begin{obs} \label{obs:add-edge-good-embedding}
		Let $m, d$ be integers. Suppose that $\varphi : H \to G$ is an $(m, d)$-good embedding, and let $H'$ be a graph obtained from $H$ by joining two vertices $u, v$ in $H$ for which $\varphi(u) \varphi(v)$ is an edge in $G$. Then $\varphi$, interpreted as an embedding of $H'$ into $G$, is an $(m, d)$-good embedding.
	\end{obs}

	The proof of \Cref{obs:add-edge-good-embedding} is very easy; see Lemma 5.2.8 in \cite{glebov2013dissertation} or the following proof.

	\begin{proof}
		Adding the edge $uv$ to $H$ does not change the left-hand side of \eqref{eqn:good-embedding}, and does not increase the right-hand side of \eqref{eqn:good-embedding}.
	\end{proof}

	The following lemma is a variant of Lemma 5.2.9 in \cite{glebov2013dissertation} (a similar lemma appears in \cite{draganic2021rolling}, and a stronger version is given in Lemma 3.10 in \cite{montgomery2019spanning}). As our setting is slightly different, existing variants do not seem to apply directly, so we give the proof here. Nevertheless, the proof is essentially the same.

	\begin{lem} \label{lem:add-long-path-embedding}
		Let $m, d$ be integers with $m \ge 4$. Suppose that $G$ is a $(4m-2, d+2)$-expanding graph, which is bipartite with bipartition $\{X, Y\}$, and for every two subsets $X' \subseteq X$, $Y' \subseteq Y$ of size at least $m/8$ there is an edge of $G$ between $X'$ and $Y'$.

		Let $H$ and $H'$ be bipartite graphs in $\family_{m, d, 2\log m}$ such that $H'$ can be obtained from $H$ by joining two vertices in $H$ by a path $P$ of length at least $2\log m$ whose interior vertices are not in $H$. Suppose that $\varphi : H \to G$ is a $(4m-2, d)$-good embedding. Then there exists a $(4m-2,d)$-good embedding $\varphi' : H' \to G$ that extends $\varphi$.
	\end{lem}

	\begin{proof}
		Denote the ends of $P$ by $x$ and $y$ (so $x, y \in V(H)$). 
		Write $\ell = \floor{\log m - 2}$ and $k = |P| - 2 - 2\ell$. Note that $k \ge 2$, and let $k_x, k_y$ be integers such that $k_x, k_y \ge 1$ and $k = k_x + k_y$. Let $T_x$ be a rooted tree constructed as follows. Let $S_x$ be a complete binary tree of height $\ell$ rooted at $r_x$, let $P_x$ be a path of length $k_x$ with ends $u_x$ and $r_x$, such that the only common vertex of $P_x$ and $S_x$ is $r_x$. Take $T_x$ to be the tree $S_x \cup P_x$, rooted at $u_x$. Define $T_y$ similarly, using a binary tree $S_y$ of height $\ell$ and a path $P_y$ of length $k_y$ (see \Cref{fig:Tx}). 
		\begin{figure}[h]
		    \centering
		    \includegraphics[scale = .6]{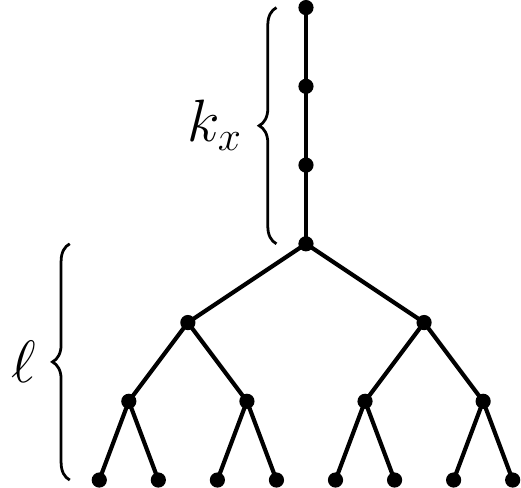}
		    \caption{The rooted tree $T_x$}
		    \label{fig:Tx}
		\end{figure}
		
        Consider the graph $F$ obtained from $H$ by attaching $T_x$ to $x$ at the root $u_x$, and attaching $T_y$ to $y$ at the root of $T_y$ (so that $x$ is the root of $T_x$ and $y$ is the root of $y$; the other vertices of $T_x$ and $T_y$ are not in $H$; and the trees $T_x$ and $T_y$ are vertex-disjoint). 

		Observe that $F$ can be obtained from $H$ by successively adding vertices of degree $1$, it has maximum degree at most $d$ (for this we use that $x$ and $y$ have degrees at most $d-1$ in $H$ by the assumptions on $H$ and $H'$, $k_x, k_y \ge 1$, and $d \ge 3$), and we can bound the number of vertices in $F$ as follows. 
		\begin{equation*}
			|F| \le |H| + k + 2(1 + 2 + \ldots + 2^{\ell}) \le |H'| + 4 \cdot 2^{\ell} \le 2m.
		\end{equation*}
		Thus by \Cref{thm:extend-good-embedding} there is a $(4m-2, d)$-good embedding $\psi : F \to G$ that extends $\varphi$. Denote by $L_x$ and $L_y$ the images of the non-root leaves of $T_x$ and $T_y$ under $\psi$. Because $H'$ and $G$ are bipartite and $H$ is connected, without loss of generality $L_x \subseteq X$, and thus $L_y \subseteq Y$.
		As $|L_x|, |L_y| = 2^{\ell} \ge m/8$ and by assumption on $G$, there exist vertices $x' \in L_x$ and $y' \in L_y$ such that $x'$ and $y'$ are adjacent in $G$. Consider the graph $F'$ obtained from $F$ by removing all vertices in $T_x$ except for the vertices on the path from $\psi^{-1}(x')$ to the root $x$, and similarly removing all vertices in $T_y$ except for the vertices on the path from $\psi^{-1}(y')$ to $y$. By \Cref{lem:trim-good-embedding}, the restriction $\psi'$ of $\psi$ to $F'$ is $(4m-2, d)$-good. Note that $H'$ can be obtained from $F'$ by joining $\psi^{-1}(x')$ and $\psi^{-1}(y')$. Thus, by \Cref{obs:add-edge-good-embedding} and by choice of $x'$ and $y'$, the embedding $\varphi' : H' \to G$ obtained by thinking of $\psi'$ as an embedding of $H$ is $(4m-2, d)$-good.
	\end{proof}

	The proof of \Cref{thm:embed-H} follows easily from the above results.

	\begin{proof}[Proof of \Cref{thm:embed-H}]
		Let $H$ be a bipartite graph in $\family_{m, d, 2\log m}$, and let $H_0, \ldots, H_t = H$ be such that $H_0$ is a graph on one vertex, and $H_{i+1}$ can be obtained from $H_i$ by either adding a new vertex and joining it by an edge to some vertex in $H_i$, or by connecting two vertices in $H_i$ by a path of length at least $2 \log m$ whose interior vertices are new. By \Cref{obs:single-vx-embedding}, any embedding $\varphi_0 : H_0 \to G$ is $(4m-2, d)$-good. Suppose that $\varphi_i : H_i \to G$ is a $(4m-2, d)$-good embedding. If $H_{i+1}$ can be obtained from $H_i$ by joining a new vertex by an edge to $H_i$, then by \Cref{thm:extend-good-embedding} there is a $(4m-2, d)$-good embedding $\varphi_{i+1} : H_{i+1} \to G$. Otherwise, there exists a $(4m-2, d)$-good embedding $\varphi_{i+1} : H_{i+1} \to G$ by \Cref{lem:add-long-path-embedding}. It follows that there is a good embedding $\varphi : H \to G$, as required.
	\end{proof}
\section{Embedding powers of subdivisions into blowups of large subdivisions} \label{sec:subdivisions}

	In this section we prove \Cref{lem:reduction}. 
	We split the task of proving the lemma into the following two propositions.

	\begin{prop} \label{prop:reduction-square}
		Let $d$ and $\ell$ be integers, with $\ell \ge 6$, and set $D = 4d^3$ and $L = (\ell-3)/2$.
		Let $H \in \family_{n, d, \ell}$. 
		Then there is a graph $F \in \family_{n, D, L}$ and a map $\varphi : H \to F$ such that 
		\begin{itemize}
			\item
				$\dist_F(\varphi(x), \varphi(y)) \le \ceil{\dist_H(x, y)/2}$ for every $x, y \in V(H)$, and
			\item
				$|\varphi^{-1}(u)| \le 4d^2$ for every $u \in V(F)$.
		\end{itemize}
	\end{prop}

	\begin{prop} \label{prop:reduction-subdivision}
		Let $d$ and $\ell$ be integers, with $\ell \ge 7$, and set $D \geq 4d^2$ and $L = (\ell-6)/2$. Let $H \in \family_{n, d, \ell}$. Then there exists a bipartite graph $F$ in $\family_{n, D, L}$ and a map $\varphi : H \to F^-$ such that
		\begin{itemize}
			\item
				adjacent vertices in $H$ are mapped either to the same vertex or to adjacent ones in $F^-$, and 
			\item
				$|\varphi^{-1}(u)| \le 4d$ for every vertex $u \in V(F)$.
		\end{itemize}
	\end{prop}

	We now prove \Cref{lem:reduction} using the two propositions, whose proofs we delay to subsequent subsections.

	\begin{proof}[Proof of \Cref{lem:reduction} using \Cref{prop:reduction-subdivision,prop:reduction-square}]
		Let $s = \ceil{\log_2 t}$. Define $\ell_0, \ldots, \ell_s$ as follows.
		\begin{equation*}
			\ell_0 = \ell \quad \text{and} \quad \ell_i = \floor{(\ell_{i-1}-3)/2} \text{ for $i \in [s]$}. \\
		\end{equation*}
		Define $d_0, \ldots, d_s$ and $b_1, \ldots, b_s$ as follows.
		\begin{equation*}
			d_0 = d \quad \text{and} \quad d_i = 4d_{i-1}^3, \,\, b_i = 4d_{i-1}^2 \text{ for $i \in [s]$}.
		\end{equation*}
		Now define graphs $H_0, \ldots, H_s$, such that $H_i \in \family_{n, d_i, \ell_i}$, and maps $\varphi_1, \ldots, \varphi_s$ as follows. Take $H_0 = H$. Having defined $H_{i-1}$, by \Cref{prop:reduction-square} and choice of $d_i, b_i$, there exists $H_i \in \family_{n, d_i, \ell_i}$ and a map $\varphi_i : H_{i-1} \to H_i$ such that $\dist_{H_i}(\varphi_i(x), \varphi_i(y)) \le \ceil{\dist_{H_{i-1}}(x, y)/2}$ for every $x, y \in V(H_{i-1})$, and $|\varphi_i^{-1}(u)| \le b_i$ for every $u$ in $H_i$. Set $\varphi = \varphi_s \circ \ldots \circ \varphi_1$. Then $\varphi$ is a map from $H_0 = H$ to $H_s$ such that if $\dist_H(x, y) \le 2^s$ then $\dist_{H_s}(\varphi(x), \varphi(y)) \le 1$, and $|\varphi^{-1}(u)| \le b_1 \cdot \ldots \cdot b_s$ for every $u$ in $H_s$.

		Observe that $\ell_i \ge \ell_{i-1}/2 - 2.5$ for $i \in [s]$. By iterating this we find that $\ell_s \ge 2^{-s} \ell_0 - 5 \ge \ell/2t - 5$, and so $L = \ell/4t - 6 \le (\ell_s - 6)/2$.
		Next, we have $2d_i = (2d_{i-1})^3 = \ldots = (2d_0)^{3^i} = (2d)^{3^i}$, and in particular $2d_s \le (2d)^{4t^2}$. Defining $D = (2d)^{8t^2}$, we have $D \ge (2d_s)^2$, and so by \Cref{prop:reduction-subdivision} there exists a bipartite $F \in \family_{n, D, L}$ and a map $\psi: H_s \to F^-$ such that adjacent vertices in $H_s$ are mapped either to the same vertex or to adjacent ones, and $|\psi^{-1}(u)| \le 4d_s$ for every vertex $u$ in $F$. 

		Consider the map $\rho = \psi \circ \phi$. Then $\rho$ is a map from $H$ to $F^-$ such that vertices at distance at most $t \le 2^s$ in $H$ are mapped to vertices at distance at most $1$ in $F^-$, and for every vertex $u$ in $F^-$
		\begin{equation*}
			|\rho^{-1}(u)| 
			\le b_1 \cdot \ldots \cdot b_s \cdot 4d_s 
			\le (2d_0)^2 \cdot \ldots \cdot (2d_s)^2 
			\le (2d)^{2(1 + \ldots + 3^s)} 
			\le (2d)^{3^{s+1}} 
			\le (2d)^{12t^2}.
		\end{equation*}
		Taking $T = (2d)^{12t^2}$, it follows that $H^t$ is a subgraph of $F^-[K_T]$, as claimed.
	\end{proof}

\subsection{Proof of Proposition~\ref{prop:reduction-square}}

	\begin{proof}
		Let $H_0 \subseteq \ldots \subseteq H_k = H$ be a sequence of graphs such that $H_0$ is a singleton, and $H_{i+1}$ is obtained from $H_i$ either by adding a new vertex of degree $1$, or by joining two vertices in $H_i$ by a path of length at least $\ell$ whose interior vertices are new. This sequence exists by the definition of $H$. Note that when we join two vertices by a path, we assume that we add the vertices in an order, either from $y$ to $x$ or $x$ to $y$.

		We will define a sequence of graphs $F_0 \subseteq \ldots \subseteq F_k$ such that $F_0$ is a singleton, and $F_{i+1}$ is either equal to $F_i$, or it can be obtained from $F_i$ either by adding a new vertex of degree $1$, or by connecting two vertices of $F_i$ by a path of length at least $L$ whose interior vertices are new. Additionally, we will define a sequence of maps $\varphi_i : H_i \to F_i$ for $i \in \{0, \ldots, k\}$. 

		Define $\varphi_0 : H_0 \to F_0$ to map the single vertex in $H_0$ to the single vertex in $F_0$. For $i \in \{0, \ldots, s-1\}$ define $F_{i+1}$ and let $\varphi_{i+1}$ be a map $\varphi_{i+1} : H_{i+1} \to F_{i+1}$ that extends $\varphi_i$, as follows.

		\begin{enumerate}[label = \rm(\roman*)]
			\item \label{itm:add-vertex}
				Suppose that $H_{i+1}$ is obtained by joining a new vertex $y$ to a vertex $x$ in $H_i$. Let $S = \{x\} \cup N_{H_i}(x)$. 

				If $\varphi_i(S)$ consists of a single vertex $u$, let $F_{i+1}$ be the graph obtained by adding a new vertex $v$ to $F_i$ and joining it to $u$, and define $\varphi_{i+1}(y) = v$.

				Otherwise, let $z$ be the last vertex in $S$ to join the sequence $H_0, \ldots, H_i$. Define $F_{i+1} = F_i$ and $\varphi_{i+1}(y) = \varphi_{i+1}(z)$.
			\item \label{itm:add-path}
				Suppose that $H_{i+1}$ is obtained from $H_i$ by connecting two vertices $x, y$ in $H_i$ by a path $P$ of length $h$, such that $h \ge \ell$, whose interior vertices are new. 

				Let $S_x= \{x\} \cup N_{H_i}(x)$. We would like to assume $|\varphi_i(S_x)| = 1$. If this is not the case, we add one or two new graphs $H_i', H_i''$ to obtain the sequence $H_0, H_1,\dots, H_i, H_i', (H_i'',) H_{i+1}, \dots$, and define maps $
				\varphi_i': H_i'\rightarrow F_i$ and $\varphi_i'': H_i''\rightarrow F_i$, such that either in the graph $H_i'$ or in $H_i''$ we may assume $|\varphi_i'(S_{x'})|=1$  or $|\varphi_i''(S_{x''})|=1$, where $x'$ and $x''$ are the second and third vertices in $P$ (thinking of $x$ as the first vertex). Indeed, let $H_i'$ be obtained from $H_i$ by adding $x'$ and joining it to $x$, and let $H_i''$ be obtained from $H_i'$ by adding $x''$ and joining it to $x'$. Now follow the instructions in item \ref{itm:add-vertex} to obtain a map $\varphi_i' : H_i' \to F_i$, which extends the map $\varphi_i$, i.e.\ let $z$ to be the last vertex in $S_x$ to join the sequence $H_0, H_1, \dots, H_i$ and set $\varphi_i'(x')=\varphi(z)$.
				If  $\varphi_i'(x) = \varphi_i'(x')$ then let $S_{x'} = \{x'\} \cup N_{H_i'}(x')=\{x,x'\}$. Note that  $|\varphi_i'(S_{x'})| = 1$, so  it is enough for us to consider the new sequence $H_0, H_1, \dots, H_i, H_i', H_{i+1}', \dots $ where $H_{i+1}'$ is obtained from $H_i'$ by adding a path of length $h-1$ between $x'$ and $y$.

				Otherwise, again follow item \ref{itm:add-vertex} to obtain a map $\varphi_i'' : H_i'' \to F_i$ that extends $\varphi_i'$.  Since $|\varphi_i'(S_{x'})| \neq 1$, we set $\varphi_i''(x'')=\varphi_i''(x')$. Now we have $S_{x''}=\{x'', x'\}$ and $|\varphi_i(S_{x''})|=1$, thus we may consider the new sequence $H_0, H_1, \dots, H_i, H_i', H_i'' H_{i+1}'', \dots $, where $H_{i+1}''$ is obtained from $H_i'$ by adding a path of length $h-2$ between $x''$ and $y$.

				To conclude, by possibly adding one or both of $H_i'$ and $H_i''$ to the sequence $H_0, \ldots, H_i$, we may assume that $|\varphi_i(S_x)| = 1$. Similarly, we may assume that $|\varphi_i(S_y)| = 1$, where $S_y = \{y\} \cup N_{H_i}(y)$. Taking into account the possible modifications, we now have that $h \ge \ell-4$.

				Let $h' = \floor{(h+1)/2}$, and form $F_{i+1}$ by connecting $\varphi_i(x)$ and $\varphi_i(y)$ by a path $P'$ of length $h'$ whose interior vertices are new. Extend $\varphi_i$ to a map $\varphi_{i+1} : H_{i+1} \to F_{i+1}$ such that if $P = (x_0\ldots x_h)$, where $x_0=x, x_h=y$ and $P' = (u_0 \ldots u_{h'})$ then $\varphi_{i+1}(x_{2j-1}) = \varphi_{i+1}(x_{2j}) = u_j$ for $j \in [h'-1]$, and if $h$ is even put $\varphi_{i+1}(x_{h-1}) = u_{h'-1}$.
		\end{enumerate}

		Set $F = F_k$ and $\varphi = \varphi_k$.
		We will show that 
		\begin{enumerate}[label = \rm(\alph*)]
			\item \label{itm:star}
				for every star $S$ in $H$ the image $\varphi(S)$ of $S$ consists of either a single vertex or two adjacent vertices,
			\item \label{itm:preimage}
				$|\varphi^{-1}(u)| \le 4d^2$ for every vertex $u$ in $F$,
			\item \label{itm:max-deg}
				$F \in \family_{n, D, L}$ (recall that $D = 4d^3$ and $L = (\ell - 4)/2$).
		\end{enumerate}

		To prove \ref{itm:star} we prove by induction on $i$ that $\varphi_i$ maps stars to edges or vertices. This clearly holds for $\varphi_0$. Now if the statement holds for $i$, then it is easy to see from the construction of $\varphi_{i+1}$ in terms of $\varphi_{i}$ that it holds for $i+1$, by considering all stars that appear in $H_{i+1}$ but not in $H_i$. 

		Next, to see \ref{itm:preimage}, consider a vertex $u$ in $F$. We wish to show that $|\varphi^{-1}(u)| \le 4d^2$. Let $i$ be minimum such that $u$ is in $F_i$. Note that if $u$ is the single vertex in $F_0$ then $\varphi^{-1}(u)$ contains only the single vertex in $H_0$, so we may assume that $i \ge 1$.
		Then $F_i$ is obtained by $F_{i-1}$ either by adding $u$ and joining it to a vertex in $F_{i-1}$, or by connecting two vertices of $F_{i-1}$ by a path of length at least $L$ that contains $u$.

		We assume that the former holds. It follows that $H_i$ is formed by adding a vertex $x$ to $H_{i-1}$ and joining it to a vertex $y$ in $H_{i-1}$, and $\varphi_i(x)=u$. 

		\begin{claim}\label{claim:auxPreimage} 
			$\varphi_j^{-1}(u)$ consists of a subset $S_j$ of neighbours of $y$ and a set of neighbours of $S_j$, for $j \ge i$.
		\end{claim}

		\begin{proof} 
			This clearly holds for $j = i$, because then $\varphi_i^{-1}(u) = \{x\}$. Suppose that the claim holds for $j$. There are two options. Either  $H_{j+1}$ is formed by adding a new vertex and joining it to a vertex in $H_j$, or $H_{j+1}$ is formed  by connecting two vertices in $H_j$ by a path of length at least $\ell-2$ whose interior vertices are new. In the second we have $\varphi^{-1}_{j+1}(u) = \varphi^{-1}_{j}(u)$, so it suffices to consider the first. 

			So we may assume that $H_{j+1}$ is formed by adding a new vertex $z$ and joining it to a vertex $w$ in $H_j$. We may also assume $\varphi_{j+1}(z) = u$, because otherwise $\varphi_{j+1}^{-1}(u) = \varphi_j^{-1}(u)$ and the claim follows from the assumption on $j$. Write $S_w = \{w\} \cup N_{H_j}(w)$. Let $s$ be the last vertex in $S$ to appear in $H_0, \ldots, H_j$. If $s = w$, then $w$ has exactly one neighbour $w'$ in $H_j$. Indeed, if $w$ had more than one neighbour in $H_j$, then it would have exactly two neighbours and that would mean $w$ was added as a new vertex of an internal path added to $H_j$. But note that in that case $w$ would have been mapped to a new vertex different from $u$ (which was added at an earlier step, precisely at $i$-th step), contradicting out assumption that $\varphi_{j+1}(z)=\varphi_{j+1}(w)=u$. 

			As $\varphi_{j+1}(z) = u$, by definition of $\varphi_{j+1}$, we have $\varphi_j(w) = u$ and $\varphi_j(w') \neq u$. Indeed, if $\varphi_j(w)=\varphi_j(w')$ then $\varphi_{j+1}(z)$ would be a new vertex added to $H_j$, contradicting the assumption $\varphi_{j+1}(z) = u$. By assumption on $j$ the preimage $\varphi_j^{-1}(u)$ consists of a subset $S_j$ of neighbours of $y$ and some neighbours of $S_j$. Since $w\in \varphi_j^{-1}(u)$, it follows that $w$ is either a neighbour of $y$, or a neighbour of a vertex in $S_j$. Buy $w'$ is the only neighbour of $w$, and $\varphi_j(w') \neq u$, so $w' \notin S_j$, hence the former holds and $w' = y$. This implies that $z$ is a neighbour of a neighbour of $y$, and so the statement holds for $j+1$.

			Now suppose that $s \neq w$. Let $k$ be minimum such that $s$ is in $H_k$. Then $H_k$ is formed by adding $s$ to $H_{k-1}$ and joining it to $w$ and maybe to more vertices.  Note that $u=\varphi_{j+1}(z)=\varphi_{j+1}(s)$ by definition of the mapping $\varphi_{j+1}$, hence $s\in \varphi_k^{-1}(u)$. Now if $k=i$, that means $s=x$ and $w=y$, so $z$ is a neighbour of $y$, as claimed. So we may assume $k>i$. We claim that $w$ is the only neighbour of $s$ in $H_k$. Indeed, if $s$ had other neighbours then it would have been added as a part of an internal path added to $H_k$ and  would have been mapped to a new vertex, contradicting $\varphi_{j+1}(s) = u$.  By induction, using that $i<k\leq j$, we have that $s$ is either a neighbour of $y$ or a neighbour of a vertex in $S_j$ (the set of neighbours of $y$ in $H_j$ whose image is $u$). If $s$ is a neighbour of $y$ then $y=w$, and we are done, and if $s$ is a neighbour a vertex in $S_j$, then $w$ is a neighbour of $y$, which means that also $z$ is a neighbour a vertex in $S_j$. The statement thus follows for $j+1$.
		\end{proof}

		It remains to consider the latter case. Then $F_i$ is obtained from $F_{i-1}$ by joining two vertices $v$ and $w$ by a path $P'$ of length at least $L$, and $u$ is one of the new internal vertices of this path. Then $H_i$ is obtained from $H_{i-1}$ by joining two vertices $x, y$ in $H_{i-1}$ by a path $P$ of length at least $\ell - 4$.
		Moreover, $\varphi_i^{-1}(u)$ consists of two or three consecutive vertices in the interior of $P$. Let $x', y'$ be the two vertices in $P$ that are adjacent to a vertex of $\varphi_i^{-1}(u)$. A similar argument used to prove ~\Cref{claim:auxPreimage}, would show that $\varphi^{-1}_j(u)$ consists of a set of vertices $S'$ that are neighbours of $x'$ or $y'$, and a set of neighbours of $S'$, for all $j\geq i$. 

		To summarise, $\varphi^{-1}(u)$ consists of vertices at distance at most $2$ from a set of size at most $2$, for every vertex $u$ in $F$. In particular, $|\varphi^{-1}(u)| \le 2(d+d^2) \le 4d^2$ for every vertex $u$ in $F$, as required for \ref{itm:preimage}.

		By construction, for every vertex $u$ in $F$, each of $u$'s neighbours in $F$ corresponds to a neighbour in $H$ of a vertex in $\varphi^{-1}(u)$, implying that the maximum degree in $F$ is at most $|\varphi^{-1}(u)| d \le 4d^3 = D$, proving \ref{itm:max-deg}. 

		We claim that $F$ and $\varphi$ satisfy the requirements of \Cref{prop:reduction-square}.
		The requirement on $F$ follows from \ref{itm:max-deg}, and the second itemized property follows from \ref{itm:preimage}.
		To see that the first property holds, let $x, y \in V(H)$ and let $(x_0, \ldots, x_{\rho})$, where $x_0 = x$ and $x_{\rho} = y$, be a shortest path from $x$ to $y$. By \ref{itm:star}, since $\dist_H(x_i, x_{i+2}) \le 2$ we have $\dist_F(\varphi(x_i), \varphi(x_{i+2})) \le 1$ for $i \in \{0, \ldots, \rho-2\}$. It follows that $\dist_F(\varphi(x), \varphi(y)) \le \ceil{\rho/2} = \ceil{\dist_H(x, y)/2}$, as required.
	\end{proof}

\subsection{Proof of Proposition~\ref{prop:reduction-subdivision}}
	\begin{proof} 
		Let $H_0 \subseteq \ldots \subseteq H_k = H$ be a sequence of graphs such that $H_0$ is a singleton, and $H_{i+1}$ is obtained from $H_i$ either by adding a new vertex and joining it to a vertex in $H_i$, or by connecting two vertices in $H_i$ by a path of length at least $\ell$ whose interior vertices are new.

		We will define a sequence of graphs $F_0 \subseteq \ldots \subseteq F_k$ such that $F_0$ is a singleton, and $F_{i+1}$ is either equal to $F_i$, or it is obtained from $F_i$ either by adding a new vertex of degree $1$, or by connecting two vertices in $F_i$ by a path of length at least $(\ell - 6)/2$ whose interior vertices are new. Additionally, we will define a sequence of maps $\varphi_i : H_i \to F_i^-$ for $i \in \{0, \ldots, k\}$ such that $\varphi_{i+1}$ extends $\varphi_i$. It will be convenient at times to think of the vertices of $F_i^-$ as either vertices or edges of $F_i$ in a natural way, i.e. if $w_e$ is the vertex of $F_i^-$ added via subdividing the edge $e$ of $F_i$ then we may think of $w_e$ and $e$ interchangeably.  

		Define $\varphi_0 : H_0 \to F_0^-$ to map the single vertex in $H_0$ to the single vertex of $F_0^-$.
		For $i \in \{0, \ldots, k-1\}$, we define $F_{i+1}$ and $\varphi_{i+1}$ as follows.

		\begin{enumerate}[label = \rm(\roman*)]
			\item
				Suppose that $H_{i+1}$ is obtained by joining a new vertex $y$ to a vertex $x$ in $H_i$.

				If $\varphi_i(x)$ is a vertex of $F_i$, denote $\varphi_i(x) = u$ and form $F_{i+1}$ by adding a new vertex $v$ to $F_i$ and joining it to $u$, and set $\varphi_{i+1}(y) = uv$.

				If $\varphi_i(x)$ is an edge of $F_i$, denote $\varphi_i(x) = uv$, set $F_{i+1} = F_i$ and $\varphi_{i+1}(y) = u$.

			\item
				Suppose that $H_{i+1}$ is obtained from $H_i$ by connecting vertices $x, y$ in $H_i$ by a path $P$ of length $h$, where $h \ge \ell$, whose interior consists of new vertices. 

				It will be convenient for us to assume $\varphi_i(x)$ is a vertex of $V(F_i)$. To do so, if this is not the case (i.e.\ $\varphi_i(x)$ is an edge of $F_i$) let $x'$ be the neighbour of $x$ in $P$ and let $H_i'$ be the graph formed by adding $x'$ to $H_i$ and joining it to $x$. Observe that $H_{i+1}$ can be obtained from $H_i'$ by joining $x'$ and $y$ by a path of length $h-1$ whose interior vertices are new. Follow the instructions in the previous item to extend $\varphi_i$ to a map $\varphi_i'$ from $H_i'$ to $F_i'^-$. Note that since $\varphi_i(x)=uv$ for some edge $uv\in E(F_i)$, we set $F_i'=F_i$ and $\varphi_i'(x')=u$. Now $\varphi_i'(x')$ is indeed a vertex of $F_i'$. Repeating the same reasoning for $y$, with some abuse of notation we may assume that $\varphi_i : H_i \to F_i^-$ maps $x$ and $y$ to vertices of $F_i$, and $H_{i+1}$ is obtained from $H_i$ by joining $x$ and $y$ by a path $P$ of length $h \ge \ell-2$.

				Let $h' \in \{\floor{h/2}, \floor{h/2}-1\}$ be such that the graph $F_{i+1}$, obtained from $F_i$ by joining $\varphi_i(x)$ and $\varphi_i(y)$ by a path $P'$ of length $h'$ whose interior consists of new vertices, is bipartite (such $h'$ and $F_{i+1}$ exist as $F_i$ is bipartite, the choice of $h'$ depends on whether $\varphi_i(x), \varphi_i(y)$ are on the same side of the bipartition or not). Let $\varphi_{i+1} : H_{i+1} \to F_{i+1}^-$ be a map that extends $\varphi_i$, such that the vertices in the interior of $P$ are mapped to edges of $P'$ or vertices in the interior of $P'$, each edge of $P'$ is the image of exactly one vertex in $P$, each vertex in the interior of $P'$ is the image of one or two vertices of $P$, and consecutive vertices of $P$ are mapped either to the same vertex of $F_i^-$ or to adjacent ones. Specifically, map the second vertex of $P$ (namely, the neighbour of $x$) to the first edge in $P'$, then map the third vertex of $P$ to the second vertex of $P'$, etc., occasionally mapping two consecutive vertices of $P$ to the same vertex in $P'$, so that the penultimate vertex of $P$ is mapped to the last edge of $P'$; by choice of $h'$ such $\varphi_{i+1}$ exists.
		\end{enumerate}

		Let $F = F_k$ and $\varphi = \varphi_k$. 
		We will show that
		\begin{enumerate}[label = \rm(\alph*)]
			\item \label{itm:hom}
				adjacent vertices in $H$ are mapped either to the same vertex or to adjacent vertices in $F^-$,
			\item \label{itm:preimage-b}
				$|\varphi^{-1}(u)| \le 4d$ for every vertex $u$ in $F^-$,
			\item \label{itm:family}
				$F$ is a bipartite graph in $\family_{n, D, L}$ (recall that $D = 4d^2$ and $L = (\ell - 6)/2$).
		\end{enumerate}
		To see \ref{itm:hom}, one can prove by induction on $i$ that $\varphi_i$ maps adjacent vertices in $H_i$ to either the same vertex or adjacent ones in $F_i^-$, using the construction of $\varphi_{i+1}$ in terms of $\varphi_i$.

		To see \ref{itm:preimage-b}, note that every edge in $F$ is the image of exactly one vertex in $H$.
		Next, let $u$ be a vertex in $F$, and let $i$ be minimum such that $u$ is in $F_i$. Then $\deg_{F_i}(u) \le 2$, and $|\varphi_i^{-1}(u)| \le 2$. Every vertex in $\varphi^{-1}(u) \setminus \varphi_i^{-1}(u)$ is a neighbour (in $H$) of $\varphi^{-1}(e)$, where $e$ is one of the edges incident with $u$ in $F_i$. It follows that $|\varphi^{-1}(u)| \le 2d+2 \le 4d$, as required for \ref{itm:preimage-b}.

		For \ref{itm:family}, let $u$ be a vertex in $F$. Every neighbour of $u$ in $F$ corresponds to a neighbour (in $H$) of some vertex in $\varphi^{-1}(u)$. It thus follows that the degree of $u$ in $F$ is at most $d \cdot |\varphi^{-1}(u)| \le 4d^2 = D$, using \ref{itm:preimage-b}. Property \ref{itm:family} follows by construction of $F$.

		The proof of \Cref{prop:reduction-subdivision} follows directly from \ref{itm:hom}, \ref{itm:preimage-b} and \ref{itm:family}.
	\end{proof}

\end{document}